\documentclass[11pt]{article}
\usepackage{amsmath, amscd,amsbsy, amsthm, amsfonts, amssymb}
\usepackage{rotating}
\usepackage{pst-node}
\usepackage{tikz}
\usepackage{tikz-cd}\usepackage{dcolumn,multicol,url}
 \newtheorem{thm}{Theorem}[section]
 \newtheorem{cor}[thm]{Corollary}
 \newtheorem{lem}[thm]{Lemma}
 \newtheorem{prop}[thm]{Proposition}
 \newtheorem{claim}[thm]{Claim}
 \theoremstyle{definition}
 \newtheorem{defn}[thm]{Definition}

 \newtheorem{rem}[thm]{Remark}

 \newtheorem*{propa}{Proposition 8.6}
 \numberwithin{equation}{section}

 \newcommand{\A}{\mathcal{A}}

  \newcommand{\DX}{{\rm Hom}(\Omega^{spin}_3(X), \Ar/\Zee)}
  \newcommand{\rz}{\mathbb{R}/\mathbb{Z}}
  \newcommand{\oo}{\overline}

\def\ZZ{{\mathbb Z}}

\def\RR{{\mathbb R}}

\def\D{{\bold D}}

\def\Zee{\mathbb{Z}}
\def\Z{{\mathbb Z}}

\def\Ar{\mathbb{R}}
\def\R{\mathbb{R}}
\def\Cee{\mathbb{C}}
\def\C{\mathbf{C}}

\def\cP{{\mathcal P}}

\def\ZZ{{\mathbb Z}}

\def\RR{{\mathbb R}}

\def\be{\begin{equation}}
\def\ee{\end{equation}}
\def\ba{\begin{eqneqnarray}}
\def\ea{\end{eqneqnarray}}

\def\e1{\epsilon}
\def\AAl{\mathcal{A}_{\lambda}}
\def\A0{\stackrel{\circ}{\AAl}}

\def\o1{\omega}
\def\01{\Omega}
\def\c1{\gamma}

\def\g1{\Sigma}

\def\l1{\Lambda}

\def\v1{\varphi}

\def\d1{d}

\def\f1{\frac}
\def\t1{\theta}
\def\b1{\beta}
\def\bar{\overline}
\def\bs{\begin{eqneqnarray*}}
\def\es{\end{eqneqnarray*}}
\def\m1{\Theta}
\def\w1{\wedge}

\title{The Pontrjagin Dual of $3$-Dimensional Spin Bordism}
\author{Greg Brumfiel and John Morgan}

\begin{document}
\maketitle

\section{Introduction}\label{sect1}

\subsection{Background and Statement of the Main Result}\label{sect1.1}
Motivated by his work with Gaiotto on condensed matter physics \cite{GK}, Anton Kapustin asked us
if we could compute ${\rm Hom}(\Omega^{spin}_3(B\Gamma),\Ar/\Zee)$, the Pontrjagin dual of the $3$-dimensional Spin bordism group of $B\Gamma$
for $\Gamma$ a finite group. He proposed that the group should be described in terms of triples of cochains
$(w,p,a)$, with $w\in C^3(B\Gamma;\Ar/\Zee)$, $p$ a cocycle in $Z^2(B\Gamma;\Zee/2)$ and $a$ a cocycle in $Z^1(B\Gamma;\Zee/2)$, satisfying the equation
$$dw+\frac{1}{2}p^2+\frac{1}{4}{\mathcal P}(a^2)=0.$$
Here, the function $\cP$ is the Pontrjagin square. On the cochain level, if $c$ is a $\ZZ/2$ cocycle the Pontrjagin square $\cP(c)$ is the $\ZZ/4$ cocycle given by $ C \cup C + C \cup_1 dC$, where $C$ is a canonical $\ZZ$ cochain lifting $c$ (taking only values 0 and 1 on simplices) and $\cup_1$ is the first `higher order cup product' of Steenrod. The $1/2$ and the $1/4$ in Kapustin's equation mean the maps on cochains induced by the obvious coefficient morphisms $\Z/2 \to \R/\Z$ and $\Z/4 \to \rz$.  \\

The answer we give works for any space $X$ though we find it convenient to present our work in a slightly  different form.  Namely, instead of the triples suggested by Kapustin we work with triples $(w,p,a)$ satisfying the simpler equation
$$dw+\frac{1}{2}p^2=0.$$ 
Let us give a brief discussion of the difference between Kapustin's original suggested equation and our equation.
The point is that ${\mathcal P}(a^2)$ is a $\Zee/4$-cocycle that is exact as an $\Ar/\Zee$ cocycle. In fact there is a natural way to make it exact. Namely, given a $\Zee/2$ cocycle $a$ we let $A\in C^1(X;\Zee)$ be the lift of $a$
that takes values only $0$ and $1$ on the basis of singular 1-simplices. Then $dA=2A^2$ so that $A^2$ is a canonical integral $cocycle$ lifting $a^2$.
Hence ${\mathcal P}(a^2)=A^4$, reduced mod 4. Thus, letting $1/n$  also denote the cochain maps induced by $\Z \to \rz$ with $1 \mapsto 1/n$, we have
$$d\bigl(\frac{1}{8}A^3\bigr)=\frac{1}{4}A^4=\frac{1}{4}{\mathcal P}(a^2).$$
This allows us to transform triples $(w,p,a)$ which satisfy our equation to triples
$(w-A^3/8,p,a)$ satisfying Kapustin's equation. Thus, our discussion can be transported in its entirety into
that context.\\

Here are our main results.

\begin{thm}\label{1.1} Fix a space $X$.
For any abelian group $M$ let
$C^*(X; M)$ and $Z^*(X; M)$ denote the singular cochains and singular cocycles with values in $M$.

Let
$${\bf C}(X)=\bigl(C^3(X;\Ar/\Zee)/dC^2(X;\Ar/\Zee)\bigr)\times Z^2(X;\Zee/2)\times Z^1(X;\Zee/2)$$
with its natural compact topology. (See Section~\ref{sect1.2}.)
\begin{enumerate}
\item There is a continuous multiplication on ${\bf C}(X)$ defined by  $$(w, p, a) \centerdot (v, q, b) = (u, p+q+ab, a+b),\ \ \     where$$  $$u =  w+v+\frac{1}{2}\Big[ p\cup_1q+(p+q)\cup_1(ab)+a(a\cup_1b)b\Big] + \frac{1}{4}AB^2,$$ making ${\bf C}(X)$  a (non-abelian) compact group. (See Proposition~\ref{3.8}.)  ($A$ and $B$ are the canonical integral lifts of $a$ and $b$ discussed in the previous paragraph.)
\item The function 
$${\bf D}\colon {\bf C}(X)\to C^4(X;\Ar/\Zee)$$
defined by
$${\bf D}(w,p,a)=dw+\frac{1}{2}p^2$$
is a continuous homomorphism. (See Proposition~\ref{3.2}.)
\item Define ${\bf C'}(X)=C^1(X;\Zee/2)\times C^0(X;\Zee/2)$ with multiplication
$$(t,x)\cdot (s,y)=(t+s+xdy,x+y).$$
This is a compact group and the map
$${\bf D'}\colon {\bf C'}(X)\to {\bf C}(X)$$ given by
$${\bf D'}(t,x)=((1/2)tdt,dt,dx)$$
is a continuous homomorphism whose image is a normal subgroup of ${\bf C}(X)$ containing the commutator subgroup. (See Propositions~\ref{3.7} and~\ref{3.10}.)
\item  Furthermore, ${\bf D}\circ {\bf D'}=0.$ (See Claim~\ref{DD'claim}.)
\end{enumerate}
We define
$$G(X)={\rm Ker}({\bf D})/{\rm Im}({\bf D'}).$$
Then $G$ is a functor from the homotopy category to the category of compact abelian groups. (See Corollary~\ref{homfunct}.)
\end{thm}

We define an explicit $\Ar/\Zee$-valued pairing  between elements $(w,p,a)\in {\rm Ker}\,  {\bf D}$ and maps $f\colon M\to X$ where $M$ is a closed Spin $3$-manifold and $f$ is continuous. The value of the pairing depends only the equivalence class of $(w,p,a)$ in $G(X)$ and on the Spin bordism class of $f\colon M \to X$. This leads to the following result.\\

\begin{thm}\label{1.2}
There is an explicit continuous bilinear pairing
$$G(X)\times \Omega_3^{Spin}(X)\to \Ar/\Zee$$
whose adjoint maps each group isomorphically (and homeomorphically) to the Pontrjagin dual of the other.
Thus, the compact abelian groups $G(X)$ and ${\rm Hom}(\Omega_3^{spin}(X),\Ar/\Zee)$ are explicitly identified.
\end{thm}

Let us describe the pairing in Theorem~\ref{1.2} in a special case. Given $(w,p,a)$ representing an element of $G(X)$ and a Spin bordism element $f\colon M\to X$, suppose that $f^*p$ is of the form $\varphi^*u$ where 
$\varphi$ is a  map from  $M$ to $S^2$ and $u$ is a reduced singular cocycle (reduced in the sense of vanishes on degenerate singular simplices) on $S^2$ representing the generating cohomology class. 
Then the value of $(w,p,a)$ on $(M,f)$ is
$$\Big(\int_{[M]} f^*w\Big) \ +\ (1/2)[Spin(Z)] \ +\ (1/8){\rm Arf}(\Sigma_{f^*a})   \in \RR / \ZZ.$$
Here, $[M]$ is the fundamental cycle of a triangulation of $M$, and $Z$ is the preimage under $\varphi$ of a regular value of $\varphi$. The Spin structure on $M$ and the trivialization of the normal bundle of $Z\subset M$ induced by $\varphi$  produce a Spin structure on $Z$ and 
$[Spin(Z)]$ is the bordism class of this element, viewed as an element of $\Zee/2$. Finally,
the surface $\Sigma_{f^*a} \subset M$ is a Poincar\'e dual to $f^*a$. It inherits a $Pin^-$ structure from the Spin structure on $M$ and ${\rm Arf}(\Sigma_{f^*a})\in \Zee/8$ is the Arf invariant of a quadratic form on $H^1(\Sigma_{f^*a};\Zee/2)$ related to that structure. This Arf invariant  records the bordism class of $
\Sigma$ in $2$-dimensional $Pin^-$ bordism, a bordism group isomorphic to $\Zee/8$.\\

The most complicated parts of the argument establishing Theorems~\ref{1.1} and~\ref{1.2} are (i) showing that the product we define on ${\bf C}(X)$ is a group structure and induces an abelian group structure on $G(X)$   and (ii) showing that the pairing we define is additive in the $G(X)$-variable.
Once these are established, the proof that the pairing passes to Spin bordism and that it induces an isomorphism of Pontrjagin duals is straight-forward using natural filtrations of $G(X)$ and ${\rm Hom}(\Omega_3^{spin}(X),\Ar/\Zee)$. (See  \S \ref{sect2}). \\

It is apparent in the statement of Theorem \ref{1.1}, and even more so in the proof, that our cochain complexes must have enough structure that cup products and $\text{cup}_i$ products are defined. This is certainly the case for singular complexes, but also for simplicial complexes with vertex partial orderings so that vertices of each simplex are totally ordered. We will call such structures $ordered$ $simplicial$ $complexes$.    For example, the barycentric subdivision of any simplicial complex has this structure.  Or, one might just have a total ordering of the set of vertices. Ordered simplicial complexes form a category where the $ordered$ $simplicial$ $maps$ are those that are non-decreasing in the vertex order on simplices. 
One reason for emphasizing these other cochain complexes is that for finite complexes $X$ our group $G(X)$ is finitely computable, just as the ordinary cohomology groups of a space, along with cup products, are finitely computable given a finite triangulation.  Namely, in \S5  we prove that for ordered simplicial complexes, the groups $G$ defined using singular cochains  are naturally isomorphic to the groups defined in the same way using ordered simplicial cochains.\\

There is significant overlap of our results presented here and the mathematical parts of \cite{GK}. We thank Kapustin for arousing our interest in these questions and for many helpful conversations during the course of this work.

\subsection{Remarks on Compact Topologies}\label{sect1.2}

Let $A$ be a compact abelian topological group.  Chain groups $C_k(X, \Z)$ are always free abelian, with basis a  set of `$k$-cells'.  The cochain group $C^k(X, A) = Hom(C_k(X, \Z), A)$ is thus the direct product of copies of $A$, indexed by these $k$-cells.  As such, the cochain group is a compact abelian group.  The product topology is the weakest topology  so that the projections onto the individual factors $A$ are continuous. The differentials $d\colon C^k(X, A) \to C^{k+1}(X, A)$ are therefore continuous, since  the boundary of each $k+1$-cell involves only finitely many $k$-cells.  Thus the cocycles $Z^k(X, A)$ and the coboundaries $B^k(X, A)$ are compact groups, as are the cohomology groups $H^k(X, A)$. \\

Operations on cochains such as cup products and $\text{cup}_i$ products are continuous, when defined, since these operations are also duals of chain operations given by finite formulas on cells.   It follows that groups built from continuous operations on finite products of compact cochain sets, together with constructions of various subquotients, will all be compact groups. Moreover, whenever these constructions are functorial in $X$ at the cochain level, the maps induced by maps of spaces will also be continuous.  For example, ordered simplicial maps $X \to Y$  induce continuous cochain maps $C^k(Y, A) \to C^k(X, A)$, as well as continuous maps of cocycle groups, coboundary groups, cohomology groups, and other groups constructed from operations such as cochain products.\\

Throughout our work, all the cochain groups we consider will have these compact topologies, and all the constructions we make are continuous and keep us in the world of functors from complexes to compact sets and groups.  In particular, these considerations explain the compact topology on our group $G(X)$. Since all this is quite routine, pretty much involving nothing more than the comments in the two paragraphs above, we will not emphasize continuity and the compact topologies in most statements about $G(X)$.  Our view is that the algebra  we work out is somewhat complicated, and there is no need to clutter every statement about $G(X)$ with the topological aspects. Also, for finite complexes $X$, these topological considerations are almost trivial, as no topological groups are encountered beyond finite products of finite groups and $\R/ \Z$. In fact, for spaces with $\it{finite}$ reduced homology groups in low dimensions, such as $X = B\Gamma$ with $\Gamma$ a finite group, the groups $G(X)$ and $\Omega_3^{spin}(X)$ are also finite, so the topological content of Theorem~\ref{1.2} evaporates.\\

\subsection{Comparison with Homotopy Theoretic Approaches}\label{sect1.3}
While our approach to $\DX$ is direct and ``hands on," there are other methods and results in algebraic topology and homotopy theory that are at least conceptually relevant for the computation of Spin bordism in general, and 3-dimensional Spin bordism in particular. For completeness, we briefly discuss some of these methods. We will discuss the Atiyah-Hirzebruch spectral sequence in Section~\ref{sect2.1} below.  In low dimensions the differentials in this spectral sequence are known, and hence this computes Spin bordism up to the determination of certain group extensions, which are not at all transparent. Our results do determine these extensions in dimension 3. In higher dimensions, the differentials in the Atiyah-Hirzebruch spectral sequence become difficult to compute, as they correspond to various higher order cohomology operations.\\

 Spin bordism is represented by the spectrum $MSpin$ and the reduced Spin bordism $\widetilde{\Omega}^{spin}_j(X)$ is the stable homotopy group $\pi_j(X \wedge MSpin)$.  Since $\Omega^{spin}_3(pt) = 0$, reduced and unreduced Spin bordism coincide in dimension 3. Since the inclusion of the sphere spectrum $S \to MSpin$ is 3-connected, it is easy to see $\widetilde{\Omega}^{spin}_j (X) \simeq \widetilde{\Omega}^{framed}_j(X)$, the reduced framed bordism, for $j\leq 3$. Also, the $K$-theory orientation $MSpin \to bO$ is 8-connected, so $\widetilde{\Omega}^{spin}_j(X) \simeq \widetilde{kO}_j(X)$ for $j \leq 8$.  Here, $kO$ is 0-connected real $K$-theory, represented by the 0-connected cover $bO$ of the $BO$ spectrum. The full spectrum $MSpin$ is a sum of $bO$ and Eilenberg-MacLane spectra $k(\Z/2)$ of various stable dimensions and various  more highly connected covers of $bO$.\\

Classical Adams spectral sequence methods in homotopy theory allow sophisticated computations of  connective $K$-theory for a wide range of dimensions for many spaces $X$.  See for example \cite{Bruner} when $X = B\Gamma$, the classifying space of a finite group. From the known description of $MSpin$, these methods thus also compute Spin bordism in a wide range of dimensions.  But the typical output of these methods is a list of abelian groups. It is often not at all clear how to translate a specific Spin bordism representative $M \to X$ into an element of one of these groups, even in low dimensions. So there is something more constructive about our answer than answers produced by classical algebraic topological methods in stable homotopy theory.  On the other hand, our constructions are quite technical, and we would be hard pressed to extend them to higher dimensions.  Dimensions 1 and 2 are easy enough, and it is plausible that we could carry out a discussion in dimension 4.  But that's probably about it.\\

As mentioned above, our group $G(X)$ will be constructed from equivalence classes of tuples of cocycles and cochains that satisfy certain relations.  Before getting to the complicated details, we want to try to offer a sort of conceptual explanation of what is  happening behind the scenes.\\

The Brown classifying space of an  abelian group-valued homotopy functor is a homotopy commutative and associative H-space.  For an ordinary cohomology group, the classifying space is just an Eilenberg-MacLane space $K(M, m)$.  A simplicial model has as $q$-cells the $M$-valued $m$-cocycles on the standard $q$-simplex. So an explicit simplicial map $X \to K(M,m)$ is just a cocycle on $X$. In the simplicial world, a homotopy between such cocycle maps is interpretable as writing the difference of the cocycles as a coboundary. So one obtains $H^m(X; M) = Z^m(X)/ B^m(X) = [X, K(M,m)]$. That is, homotopy classes of maps identify with cocycles mod coboundaries; i.e., cohomology.\\

Next, consider a classifying space with a two-stage Postnikov tower $E$ with a $k$-invariant
$K(M,m) \to K(N,n)$, realized by an actual map, not just a homotopy class. Now a simplicial map $X \to E$ can be interpreted as a cocycle $X \to K(M,m)$ together with a choice of lifting to $E$, which means the cocycle  determined by the composition $X \to K(M,m) \to K(N,n)$ is written as a coboundary.  The $H$-space structure can be represented by a map $E \times E \to E$, which is more cocycle and coboundary data, as is associativity and commutativity of the $H$-space structure up to homotopy. Homotopy classes of maps $X \to E$ become equivalence classes of collections of cocycles and cochains satisfying  some coboundary relations, and the product operation is described by other cochain and cocycle formulas.\\

For a more complicated Postnikov tower, extending this viewpoint would not be manageable.  But our construction of the Pontrjagin dual of three-dimensional Spin bordism  is exactly this sort of description of homotopy classes of maps to a classifying space which is a relatively simple three-stage Postnikov tower. We will not explicitly deal with the Brown representing space for our functor $G(X)$, but  the comments in these last paragraphs do put the constructions we actually make in a broader context. Strictly speaking, this Brown representing space represents a compact functor, which would add another layer of complexity to the above discussion of  Postnikov systems.\\

Even for a very simple Postnikov tower, the $H$-space structure can exhibit complications.  For example, consider the abelian group valued functor $$ H^2(X; \Z/2) \times H^1(X; \Z/2)\ \text{with product}\ (p,a)(q,b) = (p+q+ab, a+b).$$  Its classifying space $E$ is just $K(\Z/2 , 2) \times K(\Z/2, 1)$, but the H-space product $E\times E\to E$ will include a representative for the cup product of cocycles, $$K(\Z/2, 1) \times K(\Z/2, 1) \to K(\Z/ 2, 2).$$

On the cochain level, the cup product $ab$ is not commutative.  In fact, the $\text{cup}_1$ operation on cocycles is exactly built to satisfy $d(a \cup_1 b) = ab + ba.$ (With $\Z/2$ coefficients, we can ignore $\pm$ signs.)  The operation $\text{cup}_1$ will occur in the cochain formula establishing homotopy commutativity of the H-space structure on $E$.  As a matter of fact, this simple example occurs as part of our construction of $G(X)$, namely it is a description of a natural quotient of $G(X)$.\\

\section{The Atiyah-Hirzebruch Spectral Sequence and Filtrations}\label{sect2}

\subsection{The Spectral Sequence and Filtration of $\Omega_3^{spin}(X)$}\label{sect2.1}
In Theorems~\ref{1.1}  and ~\ref{1.2} we allowed $X$ to be any space.  For some arguments we need $X$ to be a complex. Any space $X$ is weakly homotopy equivalent to the geometric realization of its singular complex.  Weak homotopy equivalences always induce isomorphisms of functors like bordism.  Thus, it is no loss of generality to assume $X$ is a complex of some kind (e.g. simplicial, cell, CW, semi-simplicial.)\\

To establish the isomorphism between $G(X)$ and the Pontrjagin dual of Spin bordism of $X$, we compare natural filtrations.
The $E^2_{i,n-i}$ term of Atiyah-Hirzebruch spectral sequence for $\Omega^{spin}_n(X)$ is identified with the homology group 
$H_i(X;\Omega^{spin}_{n-i})$. The diagonal line in total degree 3 on the $E^2$ page thus has entries $$E^2_{3,0} = H_3(X;\Zee),\  E^2_{2,1} = H_2(X;\Z/2),\  E^2_{1,2} = H_1(X;\Z/2),\  E^2_{0,3} = 0.$$  These coefficient groups correspond to the low dimensional Spin bordism groups, which in degrees $0, 1, 2, 3$ are $\Z, \Z/2, \Z/2, 0$. The $E^\infty$ terms on the degree 3 diagonal are the successive  quotients of an increasing filtration $$0 \subset
F_1(X)\subset F_2(X)\subset F_3(X) =\Omega^{spin}_3(X)$$
where $F_j(X)$ is   ${\rm Im}  \bigl(\Omega^{spin}_3(X^{(j)}) \to\Omega^{spin}_3(X)\bigr)$ where $X^{(j)}$ is the $j$ skeleton of $X$.\\

$\Ar/\Zee$ is injective, and the exact functor ${\rm Hom}(*, \R/\Z)$ converts the Atiyah-Hirzebruch spectral sequence for $\Omega_*^{spin}(X)$ to a cohomology spectral sequence for the Pontrjagin dual of $\Omega_*^{spin}(X)$. In the dual spectral sequence the groups on the degree 3 diagonal on the $E_2$ page are $$E_2^{3,0}=H^3(X;\R/\Z),\  E_2^{2,1}=H^2(X;\Z/2),\  E_2^{1,2}=H^1(X;\Z/2),\  E_2^{0,3}=0.$$ 
There is one differential in this dual spectral sequence involving the degree 3 diagonal, namely, for $[p] \in H^2(X;\Z/2) = E_2^{2,1}$ one has $d_2([p]) = (1/2)[p]^2 \in H^4(X;\R/\Z) = E_2^{4,0}$.  Basically, this differential corresponds to the first $k$-invariant of $MSpin$.  We will denote by $SH^2(X;\Z/2)$ the kernel of this differential.  That is, $SH^2(X;\Z/2)$ consists of those classes $[p] \in H^2(X;\Z/2)$ so that $(1/2)[p]^2 = 0 \in H^4(X;\R/ \Z)$. An equivalent condition is that $[p]^2$ is the $\ZZ/2$ reduction of a $torsion$ integral class in $H^4(X; \ZZ)$. \\

As far as the degree 3 diagonal is concerned $E_3 = E_\infty$. The groups on the $E_\infty$ page occur as successive quotients of the groups in a filtration
$$
Hom(\Omega^{spin}_3(X),\R/\Z) =F^0(X) \supset F^1(X) \supset F^2(X) \supset 0,$$
where $F^i(X)={\rm Ann}(F_i(X))$. Thus
\[ \begin{array}{ccc}
F(X) / F^1(X)  & = & H^1(X; \ZZ / 2)\\
F^1(X)/ F^2(X) & = &SH^2(X; \ZZ /2)\\
 F^2(X) &= &H^3(X; \RR / \ZZ) \end{array}\]

The $E_2$ and $E_3$ pages of the spectral sequence involve cohomology groups, but cochain groups underlie everything, back on the $E_1$ page. So, lurking in $E_1$ are cochains in $C^3(X;\R/\Z), C^2(X;\Z/2)$, and $C^1(X;\Z/2)$.\\

\subsection{The Filtration of $G(X)$}\label{sect2.2}

Here is a description of the natural filtration of $G(X)$.

 \begin{prop}\label{gradeds}\label{2.1}
There is a natural filtration on $G(X)$:
$$G(X)=G^0(X)\supset G^1(X)\supset G^2(X)\supset 0$$
where $G^1(X)$ and $G^2(X)$ are the compact subgroups consisting of all elements of $G(X)$ represented by triples of the form $(w,p,0)$ and $(w,0,0)$, respectively.
Then:
\begin{enumerate}
\item the map $(w,p,a)\mapsto [a]$ defines a natural  isomorphism of compact abelian groups (see Proposition~\ref{4.3})
$$G(X)/G^1(X)\to H^1(X;\Zee/2),$$ 
\item the map $(w,p,0)\mapsto [p]$ defines a natural isomorphism of compact abelian groups (see Proposition~\ref{4.2})
$$G^1(X)/G^2(X)\to SH^2(X;\Zee/2),$$

\item  the map $(w,0,0)\mapsto [w]$ defines a natural isomorphism of compact abelian groups (see Proposition~\ref{4.1})
$$G^2(X)\to H^3(X;\Ar/\Zee).$$
\end{enumerate}
\end{prop}

\subsection{Comparison of the Filtrations}\label{sect2.3}

Here is the result we state in \S\ref{sectstatement} and establish in \S\ref{sectproof}, which yields Theorem~\ref{1.2}

\begin{thm}\label{dualitythm}\label{2.2}
The adjoint of the bilinear pairing given in Theorem~\ref{1.2} is compatible with the filtrations 
of  $Hom(\Omega_3^{spin}(X),\Ar/\Zee)$ and of $G(X)$ given in this section. Furthermore, with the given identifications of the associated gradeds in cohomological terms, the
induced map of the associated gradeds is the identity. Consequently, the adjoint of the pairing identifies the compact groups $G(X)$ and $Hom(\Omega_3^{spin}(X),\Ar/\Zee)$.
\end{thm}

\section{Construction of the  Group G(X)}\label{sect3}

Throughout this section the space $X$ is fixed, and we work with cochains which can be either singular cochains or simplicial cochains  of an  ordered simplicial complex. In both cases, cup products and higher order $\text{cup}_i$ products are then defined by specific cochain formulas coming from the Alexander-Whitney diagonal approximation and related higher homotopies. 
In this section, we will often suppress $X$ from our notation. Thus we will write $G$ instead of $G(X)$ and for an abelian group $M$ we will write $C^j(M), Z^j(M), B^j(M), H^j(M)$ for the cochain, cocycle, coboundary, and cohomology groups of $X$ with $M$-coefficients.

\subsection{The Definition of the Product on Triples of Cochains}\label{sect3.1}

Recall from the statement of Theorem~\ref{1.1} in the introduction that we work with triples $$(w, p, a) \in  \C = \bar{C}^3(\R/\Z) \times Z^2(\Z/2) \times Z^1(\Z/2),$$ where $$ \oo{C^3}(\R/\Z) = C^3(\R/\Z) / dC^2(\R/\Z) = C^3(\R/\Z) / B^3(\R/\Z).$$
We defined a function $\D: \C  \to C^4(\R / \Z) = \C''$ by 
$$\mathbf{D}(w, p, a) = dw + \frac{1}{2}p^2  \in C^4(\R/\Z).$$
(Obviously $\D(w+df, p,a) = \D(w, p, a)$.) \\ 

 In the statement of Theorem~\ref{1.1} we also defined a product on the set of triples $(w,p,a) \in \C$, which we repeat in Definition~\ref{3.1} below.
 We will prove that  $\D$ is `additive'. 
An important part of that product formula involves certain integral cochains, which we mentioned in the introduction.  We will discuss these integral cochains in somewhat more detail here. We use the notation $A, B\in C^1(\Z)$ to denote the unique  $\Z$ cochain lifts of $a, b \in Z^1(\Z/2)$ taking only values 0 and 1 on the basis of  $C_1(\Z)$ consisting of $1$-simplices.  We call these the $\it special$ lifts.\\

By evaluating on a $2$-simplex one sees that $dA=2A^2$ in $C^2(\Zee)$.
Thus $A^2$ is an integral cocycle, lifting the $\Z/2$ cocycle $a^2$,  and representing a torsion integral cohomology class of order 2.
The class $A^4$  then gives a canonical integral torsion lifting of the Pontrjagin square $\cP(a^2)$, which we alluded to in the Introduction. With $\R / \Z$ coefficients, $d(A^3/8) =  A^4/4 = (1/4)\cP(a^2)$.\\

The special lift of $c = a+b$ is $C = A + B + 2(A \cup_1 B)$, which follows from the fact that $(A \cup_1 B)(0,1) = -A(0,1)B(0,1)$.  One sees that $C^2 = A^2 + B^2 + d(A \cup_1B)$, since in the torsion free abelian group $C^2(\ZZ)$ we have $$2C^2 = dC = dA + dB + 2d(A \cup_1B) = 2A^2 + 2B^2 + 2d(A \cup_1B).$$
More details about special lifts are given in Section~\ref{sectcochain}.\\

\begin{defn}\label{3.1} The product of elements $(w,p,a), (v,q,b) \in \C$ is 
$$(w, p, a) \centerdot (v, q, b) = (u, p+q+ab, a+b)\ \ \     where$$  $$u =  w+v+\frac{1}{2}\Big[ p\cup_1q+(p+q)\cup_1(ab)+a(a\cup_1b)b\Big] + \frac{1}{4}AB^2$$
Here, $(1/4)$ means the coefficient map $C^*(\ZZ) \to C^*(\RR / \ZZ)$ defined by\\ $1 \mapsto 1/4 \in \RR / \ZZ$.\\
\end{defn}

\begin{prop}\label{3.2} $\bold{D} [(w,p,a)\centerdot(v,q,b)] =  \bold{D} (w,p,a) + \bold{D} (v,q, b)$.  In particular, the set of triples $\{(w,p,a) \mid \D(w,p,a)  = 0\}$ is closed under the product operation.\end{prop} 

\begin{proof} In the notation of Definition~\ref{3.1}, we need to prove $$(dw + (1/2)p^2) + (dv + (1/2)q^2) = du + (1/2)(p + q + ab)^2.$$ The key facts for computing $du$ are that $$d(p \cup_1 q) = pq + qp$$  $$d((p+q)\cup_1(ab)) = (p+q)ab + ab(p+q)$$  $$d(a(a \cup_1b) b) = a(ab+ba)b$$  $$d((1/4)AB^2) = (2/4)A^2B^2 = (1/2)a^2b^2,$$ and the (non-commutative!) expansion of $(1/2)(p+q+ab)^2$. \end{proof}

\begin{rem}\label{3.3}
A second product on $\C$ can be defined by replacing the term $+ AB^2/4$ in Definition 3.1 by $- AB^2/4$.  The map $(w,p,a) \mapsto (-w,p,a)$ defines an isomorphism between these product structures.  Thus if one product leads to a group isomorphic to $\DX$ then so will the other.  It is just a matter of changing the definition of the evaluation of a triple $(w,p,a)$ on a Spin bordism element.\\

But just changing the sign of $w$ is not very interesting.  What is far more interesting is the relation of the product formula with the evaluation of  triples $(0,0,a)$ on $f \colon M \to X$.  Associated to the cohomology class  $[f^*(a)] = \alpha \in H^1(M; Z/2)$ there is a Poincar\'e dual surface $\Sigma_{\alpha} \subset M$, and a Spin structure on $M$ gives rise to a quadratic function $q\colon H_1(\Sigma_{\alpha}; \Z/2) \to \Z/4$ refining the intersection pairing.  Such a quadratic function has a Brown-Arf invariant in $\Z/8$, and we will define our evaluation of $(0,0,a)$ on $f \colon M \to X$ to be this Arf invariant. (See \S\ref{7.1}).\\

The definition of the quadratic function $q$ is geometric and requires making a choice of whether one counts right-hand twists or left-hand twists in $M$ of certain bands around curves in $\Sigma_a$.  
Reversing this choice changes the sign of $q$ and hence of its
Brown-Arf invariant in $\Z/8$.  We follow the standard choice of many authors and count right-hand twists, see for example \cite{Brown, Guillou, KirbyTaylor, Matsumoto}.

  One of the most difficult parts of our work was showing that these two choices coincide with the two choices of products, using $+ AB^2/4$ or $- AB^2/4$ in Definition 3.1.  Still that leaves the question of which twist choice goes with which sign choice. It was only after a  calculation involving a complicated triangulation of the tangent circle bundle of $S^2$, denoted $T_SS^2$, that we were able to determine that the Brown-Arf invariant associated to counting right-hand twists, with a specific Spin structure on $T_SS^2$, is consistent with our choice of sign $+ AB^2/4$ in the product formula of Definition 3.1. \end{rem}

\begin{rem}\label{3.4}
As another comment on the product formula, we have discussed in Section~\ref{sect1.1} replacing our assumed relation $dw +(1/2)p^2 = 0$ by  Kapustin's original suggestion $dw + (1/2)p^2 + (1/4)A^4 = 0$.  This can be accomplished by replacing $(w,p,a)$ by $(w - (1/8)A^3, p, a) = 0$.  The product formula in Definition~\ref{3.1} can then be rewritten in terms of triples $(w,p,a)$ satisfying Kapustin's relation. In this form, the product becomes more complicated.  The term $(1/4)AB^2$ that occurs in the formula for $u$ in Definition~\ref{3.1} gets replaced by the more complicated term $$\frac{1}{4}[(A^2 \cup_1B^2) - (A^2+B^2)(A \cup_1 B) - (A \cup_1 B)(A^2 + B^2) - (A \cup_1 B)\ d(A\cup_1 B) )].$$
This indicates one practical advantage of our relation.
\end{rem}

\subsection{Natural Operations From Cochains to Cocycles}\label{sectkey}

In several proofs that follow we will need to show that certain cocycles constructed naturally from cochains are coboundaries. The following lemma provides a very general method for dealing with these situations.
\begin{lem}\label{key}
Fix integers $k\ge 0$ and $\ell\ge 0$ and abelian groups $M_1,\ldots, M_k$ and $N_1,\ldots, N_\ell$ as well as a group $M$.
Suppose that for each space $X$ we have a function
$$\psi_X(z_1,\ldots,z_k,x_1,\ldots,x_\ell),$$
taking values in $M$-cocycles of $X$ and 
defined for all ordered sets
$$\{z_1,\ldots, z_k,x_1,\ldots, x_\ell\}$$
consisting of singular cocycles of $z_i$ with coefficients $M_i$ and singular cochains $x_j$ of $X$  with coefficients $N_j$.
We suppose that for any continuous map $f\colon X\to Y$  we have
$$f^*\psi_Y(z_1,\ldots,z_k,x_1,\ldots,x_\ell)=\psi_X(f^*z_1,\ldots,f^*z_k,f^*x_1,\ldots,f^*x_\ell)$$ 
(that is to say $\psi$ is natural for continuous maps).
If, for each $1\le i\le k$, the cohomology class of $z'_i$  equals that of $z_i$ in $H^*(X; M_i)$  then the cohomology class of $\psi_X(z_1,\ldots,z_k,x_1,\ldots,x_\ell)$ equals that of $\psi_X(z'_1,\ldots,z'_k, 0,\ldots , 0)$.  In particular, the cocycle
$\psi_X(z_1,\ldots,z_k,x_1,\ldots,x_\ell)$ is cohomologous to both $$\psi_X(z_1,\ldots,z_k,0,\ldots,0)\ \text{and}\  \psi_X(z'_1,\ldots,z'_k,0,\ldots,0).$$
\end{lem}

\begin{proof}
Let $j_0$ and $j_1$ be the inclusions of $X$ into the $0$- and $1$-end of $X\times I$.
 Suppose given $(z_1,\ldots,z_k)$ and $(z'_1,\ldots,z'_k)$  cocycles  in $C^*(X)$
with $z_i$ cohomologous to $z'_i$ for all $1\le i\le k$ and suppose given cochains $(x_1,\ldots, x_\ell)$. Then, for $1\le i\le k$, there are singular cocycles $\hat z_i$ of $X\times I$  with
$j_0^*\hat z_i=z_i$ and $j_1^*\hat z_i=z'_i$ and for $1\le i\le \ell$, there are singular cochains $\hat x_i$ of $X\times I$ and with 
$j_0^*\hat x_i=x_i$ and $j_1^*\hat x_i = 0$ for $1\le i\le \ell$.
Then by naturality we have the following equations of cocycles
$$j_0^*\psi_{X\times I}(\hat z_1,\ldots,\hat z_k,\hat x_1,\ldots,\hat x_\ell)= \psi_X(z_1,\ldots,z_k,x_1,\ldots,x_\ell),$$
and
$$j_1^*\psi_{X\times I}(\hat z_1,\ldots,\hat z_k,\hat x_1,\ldots,\hat x_\ell)= \psi_X(z'_1,\ldots,z'_k, 0,\ldots, 0).$$
 It follows that 
$\psi_X(z_1,\ldots,z_k,x_1,\ldots,x_\ell)$ and  $\psi_X(z'_1,\ldots,z'_k, 0\ldots, 0)$ represent the same cohomology class in $H^*(X)$.
\end{proof}

\begin{rem}\label{Coops}
Although the lemma is stated very generally, it is equivalent to the same statement with specific cohomological dimensions $m_i$, $n_j$, and $m$ assigned to the cochains $z_i$, $x_j$, and $\psi_X(z_i, x_j)$.   Let $[z_i] \in H^{m_i}(X; M_i)$ denote the cohomology class of $z_i$.  Suppose $[z_i] = [z'_i]$. Then in $H^m(X; M)$  $$[\psi_X(z_1, \ldots, z_k, 0, \ldots, 0)] = [\psi_X(z'_1, \ldots, z'_k, 0, \ldots, 0)] \in H^m(X; M).$$
 Since $\psi$ is natural in $X$, it follows that $\psi$ determines a natural cohomology operation $\psi\colon \prod H^{m_i}(X; M_i) \to H^m(X; M)$.  Of course, such operations correspond to elements of $H^m( \prod K(M_i, m_i);\ M)$.\\
 
 As another application of the lemma, suppose there are no $z_i$.  Then it follows that $[\psi(x_1, \ldots, x_\ell)] = [\psi(0, \ldots, 0)] \in H^m(X; M)$ does not depend on the $x_j$.  Applying naturality to the map from $X$ to a point, one sees that $[\psi] \equiv 0$ if $m > 0$.  If $m = 0$, then $[\psi] \equiv \mu$, the constant degree zero cocycle,  for some element $\mu \in M$. (Note $H^0(X; M)$ is the set of functions from path components of $X$ to $M$.) In other words, there are no non-trivial natural operations from cochains to cohomology.
 \end{rem}

\begin{rem}
The analogous statement and proof of the lemma hold in the ordered simplicial case. 
\end{rem}

\subsection{The Definition of the Relations on Triples of Cochains}\label{sect3.2}

Here is a preview of things to come.  The product defined in \S\ref{sect3.1} on $\C$ actually induces a group structure on  $\C$. That is, the  product  is associative and all elements have inverses.  This is difficult to prove directly, but fairly easy using Lemma~\ref{key} of the previous section.  We will  form the quotient of  $\C$ by a normal subgroup.  The subgroup  is the image of another differential $\bold{D'}$ with target $\C$, a differential also defined in the statement of Theorem~\ref{1.1},   $$\C' = C^1(\Z/2) \times C^0(\Z/2) \xrightarrow{\D'} \oo{C^3}(\R/\Z) \times Z^2(\Z/2) \times Z^1(\Z/2) = \C.$$ We define a group structure on $\C'$ and prove  $\D'$ is a group homomorphism whose image contains the commutator subgroup of $\C$. This is another  result in which Lemma~\ref{key} finesses difficult direct computations.
Then $\D$ and $\D'$ will become differentials in a little  non-abelian cochain complex, $$ \C' \xrightarrow{\D'} \C \xrightarrow{\D} \C'' ,$$ whose cohomology is  an abelian group. The group $G$
is defined to be the cohomology of this little cochain complex.
\\

Here is the differential formula for $(t,x) \in \C' = C^1(\Z/2) \times C^0(\Z/2).$ $$ \bold{D'} ( t, x) = ( \frac{1}{2}t dt , \ dt,\  dx) \in \C.$$

\begin{claim}\label{DD'claim}\label{3.5}
 $\bold{D} \circ \bold{D'} = 0$.
 \end{claim}
 
 \begin{proof}
  This follows because $(1/2)d(t dt) = (1/2)(dt)^2$.
  \end{proof}
  
    The coboundaries $\D'(\C')$  will evaluate 0 on Spin bordism elements, once we explain how to evaluate $\bold{D}$ cocycles on Spin bordism elements.  This explains how we will arrive at a function $$G(X) = {\rm Ker}(\D) / {\rm Im}(\D')\to \DX.$$
    
To understand the substructure of the $\D$ cocycles generated by the image of $\bold {D'}$, it is important to define a product on the pairs $(t, x)$ so that $\bold{D'}$ preserves products.  
Here is our  product formula on $\C'$. \\

\begin{defn}\label{defn3.6} The product of elements $(t,x), (s,y) \in \C'$ is defined by
$$ ( t, x) \centerdot (s,y) = (  t + s + xdy,\ x + y)$$  \end{defn}

\begin{prop}\label{3.7} $\C'$ is a group and $\D'$ preserves products. That is,
$$\bold{D'}[ ( t, x) \centerdot ( s, y) ] = \bold{D'}( t, x) \centerdot \bold{D'} ( s, y).$$\end{prop}

\begin{proof} Associativity of the product $(t,x)(s,y) = (t+s+xdy,\ x+y)$ on $\C'$ is straightforward.  We also have inverses in $\C'$, specifically one has $(t,x)(t + x dx, x) = (0,0)$.\\

We turn to the statement that $\D'$ preserves products.  Equality of the two sides in the third $Z^1(\Z/2)$ coordinate is trivial, both are $dx+dy$.  Almost as easy is equality in the second $Z^2(\Z/2)$ coordinate, since $d(t+s+xdy) = dt+ds+dxdy$. The proof we give that the first coordinates in $\oo{C^3}(\R/ \Z)$ of the two sides of the equation in Proposition~\ref{3.7} coincide relies on Lemma~\ref{key}. What we will prove  is that when we write out the difference of the  first coordinates in this equation, we find that it is a coboundary $dh \in C^3(\R/\Z)$.  But coboundaries  $ dh$ are trivial in $\oo{C^3}(\R/\Z).$\\

Write $$\bold{D'}[ ( t, x) \centerdot ( s, y) ] = (C_1(t,s, x,y), dt+ds+dxdy, dx+dy)$$ and $$\bold{D'}( t, x) \centerdot \bold{D'} ( s, y) = (C_2(t,s, x,y), dt+ds+dxdy, dx+dy).$$

Both $C_1$ and $C_2$ are natural operations defined on cochains $t,s,x,y$.  For $C_1$, from the definition of $\D'$  (above Claim~\ref{3.5}) and the definition of the product in $\C'$  (Definition~\ref{defn3.6}), one just sees sums of products of the cochains and their differentials, along with the coefficient morphism $(1/2)\colon H^*(\ZZ / 2) \to H^*(\RR / \ZZ)$. For $C_2$, from the product formula in $\C$  (Definition~\ref{3.1}) one also sees terms involving $\text{cup}_1$ products and the term $(1/4)(Dx)(Dy)^2$, where $Dx$ and $Dy$ are the special integral lifts of $dx$ and $dy$.  But the special integral lift is a natural cochain operation $C^*(\ZZ / 2) \to C^*(\ZZ)$.\\

Arbitrary elements in ${\rm Im}(\D')$ are $\D$-cocycles and arbitrary products of $\D$-cocycles are $\D$-cocycles, by Claim~\ref{3.5} and Proposition~\ref{3.2}. Thus $$dC_1 = (1/2)(dt + ds + dxdy)^2 = dC_2.$$  It follows that $$\psi(t,s,x,y) = C_1(t,s,x,y) - C_2(t,s,x,y)$$ is a natural operation from cochains (no cocycle variables) to 3-dimensional cocycles with $\RR /\ZZ$ coefficients.  By Remark~\ref{Coops} following Lemma~\ref{key}, the cohomology class $[\psi(t,s,x,y)] = 0 \in H^3(\RR / \ZZ)$ for all $t,s,x,y$.  But this means exactly that $C_1 - C_2$ is always a coboundary and hence is zero in $\overline{C}(\Ar/\Zee)$. 
\end{proof}

\subsection{Discussion of Associativity and Inverses}\label{sect3.3}

Now we turn to the proof that ${\bold C}$ is a group.

\begin{prop}\label{3.8}  The product on $\C $  is associative and defines a group structure on this set. The identity element is $(0,0,0)$ and the inverse of $(w,p,a)$ is $$(w,p,a)^{-1} = (-w + \frac{1}{2}p\cup_1a^2+ \frac{1}{4}A^3,\ p+a^2,\ a).$$  Also, $$(w,p,a)^{-1} = (w,p,a)^3(-4w+ \frac{1}{2}a^3,\ 0, \ 0).$$\end{prop}

\begin{proof} For the inverse statement, we write out the product and do a little canceling

$$(w,p,a)(-w + \frac{1}{2}p\cup_1a^2+ \frac{1}{4}A^3,\ p+a^2,\ a) = $$
$$(\frac{1}{4}A^3 + \frac{1}{2}[(p\cup_1p  + a(a\cup_1a)a]+ \frac{1}{4}A^3,\ 0,\ 0).$$
But by (SL4) from  Section~\ref{sectcochain}, we have $(1/2)p\cup_1p = d(P/4) $,  which is 0 in $\overline{C}^3(\R / \Z)$.  Here, $P$ is the special integral lift of the $\Z/2$ cocycle $p$. (Alternatively, $p \cup_1 p$ represents $Sq^1(p)$, which is the $\ZZ/ 2$ reduction of the integral Bockstein of $p$, which is a torsion integral class.)\\

Next, $(1/2)a(a\cup_1a)a = (1/2)a^3$, since $a\cup_1a = a$, and $(2/4)A^3 = (1/2)a^3$, so $$\frac{1}{4}A^3 + \frac{1}{2}a(a\cup_1a)a + \frac{1}{4}A^3 = 0 \in C^3(\R/\Z),$$ 
 proving the inverse statement.\\

Another way to find the inverse (after associativity is established!) is to note $$(w,p,a)(w,p,a) = (2w + \frac{1}{4}A^3 + \frac{1}{2}a^3, a^2,0)= (2w- \frac{1}{4}A^3, a^2, 0).$$ So then
$$(w,p,a)^4 = (2w- \frac{1}{4}A^3, a^2, 0)(2w- \frac{1}{4}A^3,\ a^2,\ 0) = (4w+ \frac{1}{2}a^3,\ 0,\ 0).$$  Therefore $(w,p,a)^{-1} = (w,p,a)^3(-4w +(1/2)a^3,\ 0,\ 0)$.\\

For associativity, we need to compare $$[(w,p,a)(v,q,b)]\ (u,r,c)\ and\ (w,p,a)\ [(v,q,b)(u,r,c)].$$  The third coordinates are $(a+b)+c$ and $a+(b+c)$ respectively, hence coincide.  The second coordinates are $p+q+ab + r + (a+b)c$ and $p+q+r+bc+a(b+c)$, respectively, and these also coincide.
 To complete the proof of associativity  in $\C$, we need to show the first coordinates in the two ways of associating a triple product differ by a coboundary $df \in B^3(\R/\Z)$.  We will do this using Lemma~\ref{key}, but we will need the more sophisticated part of Remark~\ref{Coops} concerning cohomology operations defined on cocycles.\\
 
 Let $C_1$ and $C_2$ denote the first coordinates of the two ways of associating the triple products. Since both triple products are $\D$-cocycles, with the same second coordinate, we have  $dC_1 = dC_2$. A cursory examination of the product in Definition~\ref{3.1} shows that the only occurrences of $w,v,u$ in the two triple products are $(w+v) +u$ on the left and $w + (u+v)$ on the right.  Therefore $C_1 - C_2 = \psi(p,a,q,b,r,c) \in Z^3(\RR / \ZZ)$  is a natural function of six cocycles $p,a, q, b, r, c$. By Remark~\ref{Coops}, $[\psi]$ must be given by a natural cohomology operation; i.e., by an element  in  $H^3((K(\ZZ/2, 2) \times K(\ZZ/2, 1))^3; \ \RR / \ZZ)$.\\
 
For convenience, we will just call the universal $\ZZ / 2$ classes in the Eilenberg-MacLane spaces by the same names $p,a,q,b,r,c$. The $\RR / \ZZ$ cohomology group of the indicated product is a $\ZZ / 2$ vector space spanned by the $(1/2)$ images of the classes of degree 3 in $\ZZ / 2[p,a,q,b,r,c]$ in the cohomology of the  Eilenberg-MacLane spaces.  (In $\RR / \ZZ$ cohomology, the image of $(1/2)Sq^1$ is $(0)$. Thus $Sq^1p, Sq^1q, Sq^1r$ map to 0.  In fact,  $Sq^1(ab), Sq^1(ac), Sq^1(bc)$ also map to 0.)  So $[\psi]$ can be expressed as a degree 3 sum of products of $p,a,q,b,r,c$, with some redundancy.\\

It is quite easy to directly use Definition~\ref{3.1} to compute the two triple products if one of the pairs $(p,a), (q,b), (r,c)$ is $(0,0)$ and see that associativity holds. Therefore in these cases $[\psi] = 0$, hence the terms $$pa, pb, \ldots, rb, rc, a^2b, a^2c, b^2c, a^3, b^3, c^3$$ cannot appear in the formula for $[\psi]$.  This leaves only the term $abc$ as a possibility.  But associativity holds, hence $[\psi] = 0$, when the three elements $(w,p,a), (v,q,b), (u,r,c)$ coincide.   So it cannot be that $[\psi] = abc$. The only remaining possibility is that $[\psi] = 0$ and hence $C_1 - C_2$ is always a coboundary, as desired.

\end{proof}

\subsection{Discussion of Commutators and Commutativity}\label{sect3.4}

Next we take up the issue of commutativity.  First, we observe $$(w,p,0)(v,q,b) = (w+v+ \frac{1}{2}p\cup_1q,\  p+q,\ b) = (v,q,b)(w,p,0) (df,0,0)\},$$ where $f = p\cup_2q$. Thus, the elements $(w,p,0)$ commute with everything in $$\C = \oo{C^3}(\R/\Z) \times Z^2(\Z/2) \times Z^1(\Z/2) .$$ We thus have the general product computation $$ (w,p,a)(v,q,b) = (w,p,0)(v,q,0)(0,0,a)(0,0,b).$$ Also $(w,p,0)^{-1} = (-w, p, 0)$.  The commutator of two arbitrary elements $(w,p,q) = (w,p,0)(0,0,a)$ and $(v,q,b) = (v,q,0)(0,0,b)$ is then seen to be 
$$(0,0,a)(0,0,b)(0,0,a)^{-1}(0,0,b)^{-1}.$$  We also know $(0,0,a)^{-1} = ((1/4)A^3, a^2, a)$.  We have proved the first result below. \\ 

\begin{prop}\label{3.9} The most general basic commutator in $\C$ is 
$$(w,p,a)(v,q,b)(w,p,a)^{-1}(v,q,b)^{-1} = (0,0,a)(0,0,b)(0,0,a)^{-1}(0,0,b)^{-1}$$
$$= (0,0,a)(0,0,b) (\frac{1}{4}A^3, a^2, a) (\frac{1}{4}B^3, b^2, b)$$\end{prop}
 
\begin{prop}\label{3.10} In $C^3(\R/\Z) \times Z^2(\Z/2) \times Z^1(\Z/2)$ one has 
$$(0,0,a)(0,0,b) =(dg+ \frac{1}{2}sds, ds, 0)(0,0,b)(0,0,a) $$ with $s = a\cup_1b \in C^1(\Z/2)$ and some $g \in C^2(\R/\Z)$. Thus, in the group $$\C = \overline{C^3}(\R/\Z) \times Z^2(\Z/2) \times Z^1(\Z/2)$$ the basic commutator of Proposition~\ref{3.9}  is $$((1/2)(a\cup_1b)d(a\cup_1b), d(a\cup_1b), 0) = \D'(a\cup_1b,0).$$
In particular, the image of ${\bold D'}$ contains the commutator subgroup of ${\bold C}$.
\end{prop}

\begin{proof}
 We have $$(0,0,a)(0,0,b) = (C_1, ab, a+b)\ \text{and}\  (0,0,b)(0,0,a) = (C_2, ba, b+a)$$ for certain natural cochain functions $C_1(a, b), C_2(a,b) \in C^3(\RR /\ZZ)$ given by the product formula in $\C$ (Definition~\ref{3.1}).
Since $d(a\cup_1b) = ab+ba$, it is easy to see that with $s = a\cup_1b$ we have for some $C = C(a,b)$
$$((1/2)s ds, ds, 0)\ (0,0,b)\ (0,0,a) = (C, ab, a+b).$$
All terms are $\D$-cocycles, hence $dC_1 = (1/2)(ab)^2 = dC \in C^4(\RR / \ZZ)$. Therefore, $\psi(a,b) = C_1(a,b) - C(a,b)$ defines a natural $\RR / \ZZ$-cocycle valued function of the cocycles $a,b \in H^1(\ZZ / 2)$.  By Lemma~\ref{key},  $\psi$ must be a universal cohomology operation.  The only possibilities are $\ZZ/2$ linear combinations of the classes $$(1/2)a^3, (1/2)a^2b = (1/2)ab^2, (1/2)b^3 \in H^3(K(\ZZ /2, 1) \times K(\ZZ / 2, 1); \RR / \ZZ).$$
But if either $a$ or $b$ is 0, the commutator is 0.  Also if $a = b$ the commutator is 0.  Therefore, the only possibility is $[\psi] = 0$, which says $\psi$ is always a coboundary as asserted in the proposition.

\end{proof}

\begin{cor}\label{3.11} $$  \frac{\C}{\{\mathbf{D'} (t,0)\}}= \frac{\oo{C^3}(\RR/\ZZ) \times Z^2(\ZZ/2) \times Z^1(\ZZ/2)}{\{\mathbf{D'}(t,0)\}}$$
is an abelian group.\end{cor}

\begin{proof} The left side is the group $\C$ divided by a subgroup containing commutators. \end{proof}

\subsection{The Definition of the Group $G = G(X)$}\label{sect3.5}
We can now officially define the (compact) abelian group $G$ that is our main object of study. We have constructed a cochain complex 
$$ C^1(\ZZ/2) \times C^0(\ZZ/2) \xrightarrow{\D'} \oo{C^3}(\RR/\ZZ) \times Z^2(\ZZ/2) \times Z^1(\ZZ/2) \xrightarrow{\D} C^4(\R/\Z)$$ of (compact) non-abelian groups.\\

\begin{defn}\label{3.12} $$G =   \frac{{\rm Ker}(\D)}{{\rm Im}(\D')} = \frac{\{(w,p,a) \in \C \mid  dw + (1/2)p^2 = 0\}}{\{\D'(t,x)\}}$$\end{defn}

Since ${\rm Im}(\D')$ contains commutators it is a normal subgroup of ${\rm Ker}(\D)$ with an abelian quotient.\\

Summarizing the construction, we begin with triples $$(w,p,a) \in C^3(\RR/\ZZ) \times Z^2(\ZZ/2) \times Z^1(\ZZ/2)$$
satisfying $dw+(1/2)p^2=0$.  We have the product $$(w, p, a) \centerdot (v, q, b) = (u, p+q+ab, a+b)\ \ \     where$$  $$u =  w+v+\frac{1}{2}\Big[ p\cup_1q+(p+q)\cup_1(ab)+a(a\cup_1b)b\Big] + \frac{1}{4}AB^2.$$
Work modulo all elements $\{(df,0,0)\} $  and get a group.  Divide further by the subgroup consisting of all elements $\{(1/2)tdt, dt, 0)\}$  and get an abelian group.  Finally divide by the subgroup consisting of all elements $\{(1/2)tdt,\ dt,\ dx)\}$ and that produces the group $G$.

\section{ A Filtration of Group $G$}\label{sect4}
\subsection{The Filtration and the Successive Quotients}\label{sect4.1}
Define a filtration of our group, $G  \supset G^1 \supset G^2 \supset 0$, by
 $$G = \{(w,p,a)\} \supset
G^1 = \{(w,p,0)\} \supset G^2 = \{(w,0,0)\} \supset (0,0,0).$$
The notation here for the subgroups means the elements of $G$ $\it{ represented}$ by triples $(w,p,a)$ with  0 entries where indicated.  Now it is trivial from what we have done that $G^1$ and $G^2$ are subgroups. The important point is to clarify their structure and compute the successive quotients in the filtration.\\

 We will see that the successive quotients are exactly the same as the successive quotients in Section~\ref{sect2.1} arising from the Atiyah-Hirzebruch filtration $F \supset F^1 \supset F^2 \supset 0$ of the Pontrjagin dual of $\Omega^{spin}_3(X)$. That is,  $$G^2 \cong H^3(\R/\Z)\hspace{.3in} G^1/G^2 \cong SH^2(\Z/2) \hspace{.3in} G/G^1 \cong H^1(\Z/2),$$ where $SH^2(\ZZ / 2\ZZ) = \{p] \in H^2(\ZZ /2\ZZ)\ \vert\ (1/2)[p]^2 = 0 \in H^4(\RR / \ZZ)$\}.\\

Recall we work with triples $$(w, p, a) \in \C = \oo{C^3}(\R/\Z) \times Z^2(\Z/2) \times Z^1(\Z/2).$$ Triples representing elements of $G$ satisfy $\D(w,p,a) = dw+(1/2)p^2 = 0 \in C^4(\R/\Z).$ 
Triples representing the identity element of $G$ are exactly the triples of the form $((1/2)tdt, dt, dx).$\\

\begin{prop}\label{4.1}  The map $(w,0,0)\mapsto [w]$ defines an isomorphism
$G^2\to H^3(\Ar/\Zee)$.
\end{prop}

\begin{proof}  The result is immediate from the following two facts. First, 
$(w,0,0)\in {\rm Ker}(\mathbf{D})$ if and only if $dw=0$. Second,
$(w,0,0)\sim (0,0,0)\in G^2$ if and only if it is in the image of ${\bold D'}$ if and only if
$(w,0,0)=((1/2)tdt,dt,dx)\in {\bold C}$ if and only if  $(w,0,0)=(0,0,0)\in {\bold C}$
if and only if $w=0\in \bar C^3(\Ar/\Zee)$.
\end{proof}

\begin{prop}\label{4.2} The map $G^1 /\ G^2 \to SH^2(\Z/2)$ defined by  $(w,p,0) \mapsto [p]$ is an isomorphism.\end{prop}

\begin{proof} The map is certainly well-defined since relations $((1/2)tdt, dt, 0)$ in $G^1$ and elements of $G^2$ map to 0. The map is surjective by definition of $SH^2(\Z/2)$. An element of the kernel would be $(w, dt, 0)$, with $dw +(1/2)(dt)^2 = 0$.  The element $(w,dt,0)$ is equivalent to
$$(w, dt, 0)((1/2)tdt, dt, 0) = (w+(1/2)tdt, 0,0) \in G^2.$$ \end{proof}  

\begin{prop}\label{4.3} The map $G/G^1 \to H^1(\Z/2)$ defined by $(w,p,a) \mapsto [a]$ is an isomorphism.\end{prop}

\begin{proof}The map is well-defined since relations $((1/2)tdt, dt, dx)$ and elements of $G^1$ map to 0. It is obviously surjective since $(0,0,a) \mapsto [a]$. An element in the kernel would be $(w,p,dx) = (w,p,0)(0,0,dx) \equiv (w,p,0) \in G^1$.\end{proof}

\subsection{The  Extensions for the Filtration}\label{sect4.2}

We can easily describe the structure of the extension $$ 0 \to SH^2(\Z/2) = G^1 / G^2 \to G / G^2 \to G /G^1 =  H^1(\Z/2) \to 0.$$ Consider the abelian group $SH^2(\Z/2) \ltimes H^1(\Z/2)$ defined as $$SH^2(\Z/2) \times H^1(\Z/2)\ \text{with product}\ ([p], [a]) ([q], [b]) = ([p] + [q] + [a][b], [a]+[b]).$$ Essentially the same proof as that of Proposition~\ref{4.2} gives the following.

\begin{prop}\label{4.4}The map $G/G^2 \to SH^2(\Z/2) \ltimes H^1(\Z/2)$ defined by $(w, p,a) \mapsto ([p], [a])$ is an isomorphism.\end{prop}

We next characterize the extension $$0 \to H^3(\RR/\ZZ) = G^2 \to G^1 \to G^1 / G^2 = SH^2(\ZZ/2) \to 0.$$ It follows directly from our construction of $G$ that $$G^1 \simeq \frac{\{(w,p) \in \bar{C}^3(\R/\Z) \times Z^2(\Z/2) \vert dw = (1/2)p^2\}} { \{ (1/2)tdt, dt) \} }.$$ The product is given by $(w,p)(v,q) = (w+v+(1/2)p \cup_1 q, p+q)$. \\

In general, if $V$ is a $\Z/2$ vector space then isomorphism classes of extensions $$0 \to K \to E \to V \to 0$$ correspond bijectively to homomorphisms $\alpha \colon V \to K / 2K$.  The correspondence is defined as follows. Given an extension, lift $v \in V$ to $\hat{v} \in E$ and set $\alpha(v) = 2\hat{v} \in K/2K$.  In our case, we lift $p \in SH^2(\Z/2)$ to $(w,p) \in G^1$.  We have $(w,p)(w,p) = (2w, 0)$.  If $b\colon H^3(\R/ \Z) \to  H^4(\Z/2)$ is the Bockstein for the coefficient sequence $$0 \to  \Z/2 \xrightarrow{1/2} \R / \Z \xrightarrow{2} \R / \Z \to 0, $$we have $b(2w) = p^2$. We conclude the following.

\begin{prop}\label{4.5}The characteristic homomorphism for the extension $$0 \to H^3(\RR/\ZZ) \to G^1 \to SH^2(\ZZ/2) \to 0$$ is given by $\alpha(p) = b^{-1}(p^2) \in H^3(\R/\Z) / 2H^3(\R/\Z)$. \end{prop}

\subsection{The Group $G$ as an Extension}

We now want to examine the extension $$ 0 \to G^1 \to G \to H^1(\Z/2) \to 0 $$ of Proposition~\ref{4.3} in greater detail. It turns out that the element represented as $(0,0,a)$ in our group $G$ depends only on the cohomology class $\oo{a} \in H^1(Z/2)$.  This is the content of the proposition below.\\

\begin{prop}\label{4.6}
$(0,0,a) \equiv (0,0,a+dx) \pmod {image(\bold{D'})}$. \end{prop}

\begin{proof}  Note $d(ax) = adx$. Using a relation in $G$, we get $$(0,0,a) \equiv (0,0,a) \centerdot ( (1/2)axd(ax), d(ax), dx) = (C, 0, a+dx)$$  where $C = C(a,x) \in C^3(\RR / \ZZ)$ is a natural cochain operation given in Definition~\ref{3.1}.  We want to prove $C$ is a coboundary. Since all elements are $\D$-cocycles, we have $dC = 0$. By Lemma~\ref{key}, we know that $C(a, x)$ and $C(a,0)$ are cohomologous. But if $x = 0$, obviously $C = 0$. Thus, $C(a,x)$ is always a coboundary.
\end{proof}

It follows that products $(0,0,a)(0,0,b) \in G$ only depend on $[a], [b] \in H^1(\ZZ/2)$.  We can now give a group extension cocycle $$z(\alpha, \beta):H^1(\ZZ/2) \times H^1(\ZZ/2) \to G^1$$ for the extension 
$$ 0 \to G^1 \to G \to H^1(\Z/2) \to 0 .$$  Recall in group $G$ we have the computations $$(w,p,a)(v,q,b) = (w,p,0)(0,0,a)(v,q,0)(0,0,b) = (w,p,0)(v,q,0)(0,0,a)(0,0,b)$$   $$ and\ \ \ (0,0,a)(0,0,b) = ((1/2)a(a\cup_1 b)b + (1/4)AB^2, ab, 0)(0,0, a+b) .$$ As in the previous section, we will abbreviate names for elements of $G^1$ as pairs $(w,p)$, rather than $(w,p,0)$.

\begin{prop}\label{4.7} With $(w,p), (v,q) \in G^1$ and $\alpha, \beta \in H^1(\Z/2)$, we have that the group $G$ is isomorphic to the set $G^1 \times H^1(\Z/2)$, with the product
$$((w,p), \alpha)((v,q), \beta) = ((w,p)(v,q) z(\alpha, \beta), \alpha + \beta)$$
where $z(\alpha, \beta)) \in G^1$ is represented by $((1/2)a(a\cup_1b)b + (1/4)AB^2, ab).$ Here $a$ and $b$ are any cocycle representatives for $\alpha$ and $\beta$, respectively. \end{prop}

Proposition ~\ref{4.7} yields a method for constructing (continuous) homomorphisms $G \to K$, where $K$ is a (compact) abelian group.

\begin{prop}\label{4.8} A homomorphism $\rho\colon G \to K$ is equivalent to a pair consisting  of a homomorphism $\rho^1\colon G^1 \to K$ and a function $\lambda^1\colon H^1(\ZZ/ 2) \to K$ that satisfy $$ \lambda^1(\alpha) + \lambda^1(\beta) = \lambda^1(\alpha + \beta) + \rho^1((1/2)a(a\cup_1b)b + (1/4)AB^2, ab) \in K.$$   In the notation of Proposition~\ref{4.7}, the homomorphisms are related by $$\rho((w,p), \alpha) = \rho^1((w,p)) + \lambda^1(\alpha).$$ Map $\rho$ is continuous if and only if $\rho^1$ and $\lambda^1$ are continuous.\end{prop}

It is both interesting and useful to know when an element $[a] \in H^1(\Z/2)$ lifts to an element of $G$ of order 2 or to an element of order 4.  It is easy to compute $(0,0,a)^8 = (0,0,0)$.

\begin{prop}\label{4.9} (i)  $[a] \in H^1(\Z/2)$ lifts to an element of order 2 in $G$ if and only if $[a]^2 = 0$\\

(ii) $[a] \in H^1(\Z/2)$ lifts to an element of order 4 in $G$ if and only if $[a]^2 \not= 0$ and $P([a]^2) = 2[p]^2 \in H^4(\Z/4)$ for some $p \in SH^2(\Z/2).$  (Here $P$ is the Pontrjagin square and 2 is the coefficient morphism $\Z/2 \to \Z/4$. In particular, it is necessary but not sufficient that $[a]^4 = 0$.)\end{prop}

\begin{proof}  (i) We have $(w,p,a)^2 = (2w - (1/4)A^3, a^2, 0)$.  For this element to represent $(0,0,0) \in G$, it is first necessary that $a^2 = dt$.  Conversely, if $a^2 = dt$ then $(0,0,a)^2 = ((1/2)tdt - (1/4)A^3, 0, 0) \in G$.  The element $w = (1/2)tdt - (1/4)A^3$ is an $\R/\Z$ cocycle.   If $[w] = 2[v] \in H^3(\R/\Z) $, then $(-v, 0, a)^2 = (0,0,0).$

A class $[w] \in H^3(\R/\Z)$ is divisible by 2 if and only if its integral Bockstein is divisible by 2.  Since $[a]^2 = 0$, we can write $dT = A^2 + 2Z$ for some integral 2-cocycle $Z$.   The integral Bockstein of $(1/2)tdt - (1/4)A^3$ is then represented by $(1/2)(A^2 + 2Z)^2 - (1/2)A^4 = A^2Z + ZA^2 + 2Z^2 $ which is clearly divisble by 2 in cohomology.  This competes the proof of (i).\\

(ii) $(w,p,a)^4 = (4w - (1/2)A^3, 0, 0) \in G$. Thus, given that $[a]^2\not= 0$, we can lift $[a]$ to an element of order 4 if and only if for some choice of $p \in SH^2(\Z/2)$, the integral Bockstein $\beta[2(2w) - (1/2)A^3] \in H^4(\Z)$ is divisible by 4.  This integral Bockstein is represented by $2\beta(2w) - A^4$.  The mod 2 reduction of $\beta(2w)$ is $[p]^2$. The mod 4 reduction of $2\beta(2w) - A^4$ is then seen to be $2[p]^2 - P([a]^2)$, which completes the proof of (ii).
\end{proof}

Suppose $X$ is a space with finitely generated integral homology in dimensions less than five.  It is a fair question if given enough information about low dimensional cohomology of $X$ with various coefficients, including products, is it possible to explicitly write $G(X)$ as a finite direct sum of $\R/\Z$'s and finite cyclic abelian groups. The answer is yes, but we will not carry out such an analysis, and  we emphasize that such direct  sum decompositions while explicitly computable are not natural.  The results of Propositions \ref{4.4}, \ref{4.5}, and \ref{4.7} above completely characterize all the groups considered in this section, and the extensions, in an invariant functorial manner.

\section{Functoriality Properties of the Groups $G(X)$}\label{sect5}

The construction of the group $G(X)$ is based in the singular cochain complex $S^*(X)$. The Alexander-Whitney diagonal approximation commutes with $f^*\colon S^*(Y)  \to S^*(X)$ induced by maps $f\colon X\to Y$.  All cup and $\text{cup}_i$ products also commute with $f^*$. Therefore, a simple inspection of our product operation on cochain triples and the relation map $\D'$ shows that each stage of our construction is functorial.  So we do get functorial homomorphisms $f^*\colon G(Y) \to G(X)$, for all continuous maps $f\colon X \to Y$.   Moreover, with the compact topologies induced by the compact topologies on cochain groups,  $f^*\colon G(Y) \to G(X)$ is a continuous homomorphism of compact abelian groups. Thus we have proved:

\begin{prop}\label{functorial}
$G$ is a functor from the category of topological spaces and continuous maps to the category of compact abelian groups. Furthermore, the filtration of $G(X)$ given in \S\ref{sect4} is functorial and the identifications of the associated gradeds with subgroups of cohomology given in that section are natural transformations.
\end{prop}

Although $S^*(X)$ is easy to define, a certain subcomplex of $S^*(X)$ is more natural for many purposes.  Specifically, there is the $reduced$ $singular$ $complex$ $S^*_r(X) \subset S^*(X)$, which consists of singular cochains that vanish on degenerate singular $n$-simplices $\Delta^n \to X$. (A degenerate singular simplex is one that factors through a linear projection $\Delta^n \to \Delta^k$ induced by a  non-decreasing surjective function $(0,1, ..., n) \to (0,1, ... k)$, with $k < n$.) The inclusion $S_r^*(X) \subset S^*(X)$ induces a functorial isomorphism of all cohomology groups because the subcomplex of the singular chain complex generated by degenerate singular simplices is acyclic. The inclusion $S_r^*(X) \subset S^*(X)$  also commutes with $\cup_i$ products.  Thus we get another functorial construction of a group $G_r(X)$ and a natural  homomorphism $G_r(X) \to G(X)$. We will see below that this homomorphism is an isomorphism.  \\

But one of our important goals was to produce a finitely computable theory.  So we also want to use some smaller cochain complexes defined for simplicial complexes $X$.  The Alexander-Whitney diagonal requires a partial ordering of vertices, with the property that the vertices of each simplex are totally ordered.  Then $n$-simplices are identified with the standard simplex $\Delta^n = (0,1,2,\cdots, n)$.  We  are calling simplicial complexes with this extra structure ordered simplicial complexes. Allowed maps $X \to Y$ are ordered simplicial maps, those such that on each simplex of $X$ the map on the vertices is non-decreasing. The first barycentric subdivision $X'$ of any simplicial complex always has this structure.  Vertices of $X'$ are barycenters of simplices of $X$, so they can be partially ordered by the face relation in $X$, or simply by labeling vertices of $X'$ by the dimension of simplices in $X$. The first barycentric subdivision $X' \to Y'$ of any simplicial map $X \to Y$ is an ordered simplicial map.\\

It is useful to have some notation for an ordered simplicial  structure on a space $X$.  We will write $[X]$ to mean an ordered simplicial complex.  Then $[X] \mapsto X$ becomes a forgetful functor.
As in the singular case, there are two chain complexes associated to an ordered simplicial complex $[X]$.  There is a large complex $C_*([X])$ with basis of $C_n$ given by all non-decreasing sequences of vertices $v_0 \leq v_1 \leq \cdots \leq v_n$ that span a simplex of $[X]$ of some dimension $k \leq n$. There is also the more standard complex $C_*^r([X])$ with basis of $C_n^r$ given by just the non-degenerate $n$-simplices.   There is an obvious surjection of chain complexes $C_*([X]) \to C_*^r([X])$ with kernel the subcomplex generated by degenerate simplices.  The degenerate complex is acyclic.  Passing to duals, there is an injection $C_r^*([X]) \subset C^*([X])$ with image consisting of cochains on the full complex that vanish on degenerate simplices.  Acyclicity of the degenerate complex implies $C_r^*([X]) \subset C^*([X])$ induces functorial isomorphisms of all cohomology groups. Note this situation is closely parallel to the singular complexes $S_r^*(X) \subset  S^*(X)$. 
$\text{Cup}_i$ products are defined in both combinatorial complexes and the maps preserve products.
All the constructions of our group $G(X)$ can be carried out in both complexes $C_r^*([X])$ and  $C^*([X])$. If we want absolute invariants on the simplicial category we can work with the first barycentric subdivision $X'$, with its natural partial vertex ordering.\\

If $[X]$ is an ordered simplicial complex there is a simple but important relation between the singular complexes and the combinatorial complexes.  Namely, a combinatorial simplex of $[X]$, viewed as a map $\Delta^n \to X$, $is$ an element of the singular complex of space $X$.  So there is a  tautologous inclusion of chain complexes $C_*([X]) \subset S_*(X)$.  This inclusion induces surjections of cochain complexes $S^*(X) \to C^*([X])$ and $S_r^*(X) \to C_r^*([X])$, which just amount to restricting singular cochains to the combinatorial simplices of $[X]$. By a slight abuse of notation, we will agree that if $c \in S^*(X)$ is a singular cochain, which is thus a function defined on all singular simplices of $X$, and if $[X]$ is an ordered simplicial structure on space $X$, then we will use the same symbol $c \in C^*([X])$ to designate the tautologous restriction of $c$ to combinatorial simplices of $[X]$\\

If $[X]$ is an ordered simplicial complex there is now a natural diagram of cochain complexes and continuous homorphisms

\begin{align}\label{diag1}
\begin{split}
S_r^*(X) \hspace{.3in}&\hookrightarrow\ \ \    S^*(X)   \\ 
\downarrow \hspace{.5in}& \hspace{.5in}  \downarrow \\
 C_r^*([X]) \hspace{.3in}& \hookrightarrow\ \ \   C^*([X]). 
\end{split}
\end{align}

The horizontal and vertical arrows have already been discussed.  The diagram commutes. All the arrows induce isomorphisms of cohomology groups and rings.  For the horizontal arrows, this follows from the acyclicity of the complexes spanned by degenerate simplices.  For the vertical arrows, this follows from the well known equivalence of simplicial and singular homology theories.  $\text{Cup}_i$ products are defined in all four complexes and commute with the maps.  The diagram is natural, that is, it commutes with the maps induced by ordered simplicial maps. We deduce from this diagram of cochain complexes a natural commutative diagram of compact groups and continuous homomorphisms, for any ordered simplicial complex $[X]$

\begin{align}\label{diag2}
\begin{split}
G_r(X) \hspace{.3in}&\rightarrow\ \ \    G(X)   \\ 
\downarrow \hspace{.5in}& \hspace{.5in}  \downarrow \\
G_r([X]) \hspace{.3in}& \rightarrow\ \ \   G([X]). 
\end{split}
\end{align}

Moreover, all the filtrations and isomorphisms in Section~\ref{sect4} become natural functorial parts of these four group structures. Although there may be quite a few things to check, this is all rather easy.\\

Let us reap some consequences.

\begin{prop}\label{5.1} For all spaces $X$, the top horizontal homomorphism in Diagram~\ref{diag2} is an isomorphism. For ordered simplicial complexes $[X]$, all four homomorphisms in Diagram~\ref{diag2} are isomorphisms of filtered topological groups.  Thus all four groups are identified with the singular complex version $G(X)$.  The entire diagram is functorial on suitable categories of spaces, complexes, and maps. \end{prop}

\begin{proof} Just use the  fact that our filterings $G(X) \supset G^1(X) \supset G^2(X) \supset 0$ constructed from any suitable cochain complex are natural, that is, suitable maps between cochain complexes induce filtration preserving maps of the corresponding $G$ groups.  The successive quotients in the filterings are various cohomology groups, identified in Section~\ref{sect4}. The cohomology groups, in all dimensions with arbitrary coefficients,  of the four complexes in Diagram~\ref{diag1} are naturally isomorphic.  Filtration preserving maps that induce isomorphisms on successive quotients are themselves isomorphisms, by repeated application of the 5-Lemma.\end{proof}

We point out that the cochain complex $C_r^*([X])$ is by far the smallest of the cochain complexes here, hence from a direct computational point of view for finite complexes the group $G_r([X])$ is the most palatable. The bases of various cochain groups used in the construction of $G_r([X])$ correspond to simplices of $[X]$ of the appropriate dimension.\\

Another simple filtration argument proves the following.

\begin{prop}\label{5.2}  Suppose $f: X \to Y$ is a map inducing homology isomorphisms in dimensions less than or equal to  three and a surjection on $H_4(\Zee)$.  Then $f^*: G(Y)\to G(X)$ is an isomorphism.\end{prop}

Thus,  for complexes  $X$, up to isomorphism $G(X)$ only depends on the homotopy type of the four skeleton of $X$. But we still want to know that $G(X)$ is a homotopy functor.

\begin{prop}\label{5.3}  Suppose $F = f_t\colon X\times I \to Y$ is a homotopy between maps
$$f_0, f_1\colon X \to Y.$$ Then $f_0^* = f_1^*\colon G(Y) \to G(X)$.\end{prop}

\begin{proof} The slices $i_t\colon X = X\times \{t\} \to X\times I$, along with the projection $p\colon X \times I \to X$, all induce isomorphisms of $G$'s, by Proposition~\ref{5.2}.  The compositions $$i_t^* p^*\colon G(X) \to G(X\times I) \to G(X)$$ are all identity maps. Because $p^*$ is an isomorphism, the maps $i_t^*$ must coincide, $0\leq t \leq 1$.   By looking at the compositions $F \circ i_t\colon X \to X\times I \to Y$
the proposition now follows easily. (This is a well-known argument for any functor $H(X)$ that converts homotopy equivalences to isomorphisms.)
\end{proof}

\begin{cor}\label{homfunct}
$G$ is a functor from the homotopy category to the category of compact abelian groups.
\end{cor}

\begin{cor}\label{cor5.4} Suppose $[X]$ and $[Y]$ are two ordered simplicial complexes and
$$f_0, f_1 \colon [X] \to [Y]$$
 are two ordered simplicial maps that are topologically homotopic. Then $$f_0^* = f_1^*\colon G_r([Y]) \to G_r([X]).$$\end{cor}

\begin{proof} Use Propositions~\ref{5.1} and~\ref{5.3} and the fact that the various group isomorphisms are functorial in $X$.
\end{proof}

We have seen that for an ordered simplicial complex $[X]$, the maps in Diagram~\ref{diag2} induce a canonical identification $G(X) = G_r([X])$. A variant of the corollary allows us to see that with respect to these identifications, any order preserving simplicial map $[X_1] \to [X_0]$ between triangulations of $X$ that is homotopic to the identity, induces the identity map $G_r([X_0]) \to G_r([X_1])$.  For example, this includes all order preserving maps between different triangulations of $X$ that are simplicial approximations of the identity.  Of course the cocycle and cochain arithmetic underlying  $G_r([X])$ for different combinatorial structures is very different, just as in the classical cohomology theory of complexes.  But they all yield the same group $G$ up to canonical isomorphisms.\\

In the family of  order preserving simplicial maps homotopic to the identity are the following interesting examples.   With $[X]$ an ordered complex, consider the first barycentric subdivision $[X']$ of $[X]$, with its natural vertex partial order.  Then there is a unique order preserving map $f\colon [X'] \to [X]$ such that each vertex  $v \in [X']$ maps to a vertex of the simplex $\sigma(v) $ of $[X]$ corresponding to $v$.  Namely, $f(v)$ must be the greatest vertex of $\sigma(v)$. It follows that the very explicit map $f^*\colon G_r([X]) \to G_r([X'])$ is the canonical identification obtained by comparing both to the singular $G(X)$. Also, the canonical identification of the groups defined by two different ordered structures on the same simplicial complex, is explicitly given by comparing both to the barycentric subdivision.\\

Suppose given two ordered simplicial structures $[X_0]$ and $[X_1]$ on the same space $X$.  It makes sense to ask when do two elements $(w_j,p_j,a_j) \in G_r([X_j])$ coincide under the canonical identification of the two groups?  The above paragraph gives an answer in terms of barycentric subdivisions and simplicial approximations of the identity.  But if the two simplicial structures have a common subdivision there is another answer.  Given two triangulations with isomorphic subdivisions, there exist triangulations of $I \times X$ extending the given triangulation on $\partial I\times X$.

\begin{prop}\label{5.5}  Fix two cochain representatives of elements $(w_j,p_j,a_j) \in G_r([X_j])$, $j = 0, 1$. Then these elements coincide under the canonical identifications of the groups $G_r([X_j])$ if and only if for every ordered simplicial structure $[I \times X]$ that extends the two structures $[X_0]$ and $[X_1]$ on $0 \times X$ and $1\times X$, respectively, there exists an element $(w,p,a) \in G_r([I \times X])$ so that at the level of representative cochains $(w,p,a)\vert_{[X_j]} = (w_j, p_j, a_j)$, $j = 0,1$. \end{prop} 
\begin{proof} The `if' direction is trivial, since the homotopy equivalences between $[X_0], [X_1]$ and $[I \times X]$ provide the canonical identifications of $G$-groups.  For the converse, the projection $I \times X \to X$ along with comparison of the singular theory to the simplicial theory, provides an element $(w', p', a') \in  G_r([I \times X])$ that restricts to $(w_j, p_j, a_j) \in G_r([X_j])$,  $j = 0, 1$.  We don't yet have equality at the cochain level. But if at the cochain level $(w', p', a')\vert_{[X_j]} = (w_j, p_j, a_j)((1/2)t_jdt_j, dt_j, dx_j)$, choose reduced simplicial cochains $t, x$ on $ [I \times X] $ that extend $t_j, x_j$, $j = 0,1$. Clearly $(w,p,a) = (w', p', a')((1/2)tdt, dt, dx)$ restricts exactly to $(w_j, p_j, a_j)$ on the ends $[X_j]$. \end{proof}

\section{The Pairing $G(X)\times \Omega_3^{\rm Spin}(X)\to \Ar/\Zee$}\label{sectstatement}
 The goal  of the next three sections is to produce an explicit continuous bilinear pairing
 $$G(X)\times \Omega_3^{\rm Spin}(X)\to \Ar/\Zee,$$ one whose adjoint determines a
 functorial identification
 $$\rho\colon G(X)\to \DX$$ that establishes Theorem~\ref{1.2}.
 In this section we define a continuous bilinear pairing
  $$G^1(X)\times \Omega^{spin}_3(X) \to \Ar/\Zee.$$
  whose adjoint is a continuous homomorphism
  $$\rho^1\colon G^1(X)\to \DX.$$
In the next section we define a continuous pairing
$$ H^1(X;\Zee/2)\times \Omega^{spin}_3(X) \to \Ar/\Zee$$
which is linear in the second variable (but not in the first), whose adjoint is a continuous function
$$\lambda^1\colon H^1(X;\Zee/2)\to   {\rm Hom}(\Omega^{spin}_3(X),\Ar/\Zee).$$
In Section~\ref{sectproof} we show that $\rho^1$ and $\lambda^1$ satisfy the compatibility condition given  in Proposition~\ref{4.8}.
It then follows from Proposition~\ref{4.8} that
$$\rho\colon G(X)\to {\rm Hom}(\Omega^{spin}_3(X),\Ar/\Zee)$$
defined by 
\begin{equation}\label{rhoeqn}
\rho(w,p,a)=\rho^1(w,p,0)\lambda^1(a)
\end{equation}
is a continuous homomorphism.
Invoking the filtrations from \S~\ref{sect2} we show that
$\rho$ is an isomorphism of (compact) topological abelian groups.

Here is the result that we shall establish.

\begin{thm}
The pairing
$$G(X)\times \Omega_3^{spin}(X)\to \Ar/\Zee$$ whose adjoint is $\rho$ is given as follows.
Let $(w,p,a)$ represent an element of $G(X)$ and let $f\colon M\to X$ represent an element of 
$\Omega_3^{spin}(X)$. There is an ordered simplicial map
$\varphi\colon M\to S^2$, a reduced singular cocycle $u$ on $S^2$, and a $1$-cochain $t$ on $M$ such that
$f^*p+dt=\varphi^*u$. Then,
\begin{align*}
\langle (w,p,a),(M,f)\rangle =
{\rm Arf}(M,f^*a)+\frac{1}{2} [Spin(Z)]+\int_M\Bigl[f^*w+(1/2)\bigl((f^*p\cup_1dt)+tdt\bigr)\Bigr]
\end{align*}
where ${\rm Arf}(M,\alpha)\in (1/8)\Zee/\Zee$ is Arf invariant of the $Pin^-$ bordism class of a surface in $M$ Poincar\'e dual to the cohomology class of $\alpha$ and where $Z$ is the preimage under $\varphi$ of the barycenter of a $2$-simplex in $S^2$.
\end{thm}

\begin{rem}  Note that using the notation of the theorem, we have in $G^1(M)$
\begin{align*}
(f^*w,f^*p,f^*a) &= (f^*w,f^*p,f^*a)((1/2)tdt,dt,0) \\
&= (f^*w+(1/2)(f^*p\cup_1 dt+tdt),\varphi^*u,f^*a).
\end{align*}
\end{rem}

\subsection{The Definition of $\rho^1$}\label{sect1.1'}

Let $M$ be a closed, Spin $3$-manifold. Then any $[p]\in H^2(M;\Zee/2)$ lifts to an integral class. This class is represented by a map  $ M\to \Cee P^2$
which for dimension reasons can be factored through a map $\varphi\colon M\to \Cee P^1$. 
Such a  map $\varphi$ represents an element in $\Omega_3^{spin}(\Cee P^1)\cong \Omega_1^{spin}=\Zee/2$. The identification sends $\varphi$ to the Spin bordism class of the pre-image, $Z$, of any generic value $z$ of any  map $\varphi'$ homotopic to $\varphi$ and smooth in a neighborhood of $Z$.  The Spin bordism class is denoted $[Spin(\varphi)]$ and also $[Spin(Z)]$. [This preimage is a compact $1$-manifold with a normal trivialization in $M$ induced from the normal trivialization of $z$ in $\Cee P^1$. The Spin structure on $M$ and this normal trivialization for $Z$ determine a Spin structure on $Z$.]
Thus,  $\varphi\colon M\to \Cee P^1$ records more than the integral cohomology class of $Z$; it also records the framing which is an extra element in $\Zee/2$, which is the Spin bordism class determined by the framing.

On the other hand, $\Omega_3^{spin}(\Cee P^2)=0$ and hence any map
$\varphi\colon M\to \Cee P^2$ bounds a compact Spin manifold $F\colon W\to \Cee P^2$. It follows that given a $2$-cocycle $p$ on $M$ there is a compact Spin $4$-manifold $W$ and a $2$-cocycle $P$ on $W$ such that $P|_M=p$. These remarks lead us to our two definitions of the value of a triple $(w,p,0)$
representing an element of $G^1(X)$ on a Spin bordism representative $f\colon M\to X$. Of course by naturality it suffices for each Spin $3$-manifold $M$, to give the value on an element in 
$G^1(M)$ on the identity map $M\to M$, considered as representing a Spin bordism class of $M$.

\begin{defn}\label{defn1} 
Let $M$ be a compact Spin $3$-manifold.
Any triple $(w,p,0)$ representing an element of $G^1(M)$ is equivalent to a triple $(w',p',0)$ with the property that $p'=\varphi^*u$ for some ordered simplicial map $\varphi\colon M\to \Cee P^1$ and a reduced singular cocycle $u$ on $\Cee P^1$ representing the non-trivial $\Zee/2$ cohomology class.
We set
$$\langle (w,p,0),(M,{\rm Id}) \rangle_1 = \frac{1}{2}[Spin(\varphi)]+\int_Mw'.$$ 
For $(w,p,0)$ representing an element of $G^1(X)$ and $f\colon M\to X$ representing an element of $\Omega_3^{spin}(X)$ we define
$$\langle (w,p,0),(M,f)\rangle_1=\langle(f^*w,f^*p,0),(M,{\rm Id})\rangle_1.$$
\end{defn}

The second definition was proposed by Kapustin.

\begin{defn} \label{defn2}
Let $M$ be a compact Spin $3$-manifold. Fix an ordered triangulation of $M$. Let $(w,p,0)$ represent an element of $G^1(M)$. Let $W$ be a compact Spin $4$-manifold with boundary $M$ and let $P$ be a $2$-cocycle on $W$ whose restriction to $M$ is $p$. Then
$$\langle (w,p,0),(M,{\rm Id})\rangle_2 =\frac{1}{2}\int_WP^2+\int_Mw.$$
To evaluate the integral of $P^2$ over $W$ we choose a relative fundamental cycle for $W$ extending the fundamental cycle of $M$ 
that is the sum of the $3$-simplices of the given ordered triangulation.

For $(w,p,0)$ representing an element of $G^1(X)$ and $f\colon M\to X$ representing an element of $\Omega_3^{spin}(X)$ we define
$$\langle(w,p,0),(M,f)\rangle_2=\langle(f^*w,f^*p,0),(M,{\rm Id})\rangle_2.$$
\end{defn}

It is not clear at this point that either definition gives a well-defined pairing of $G^1(X)$ with $\Omega_3^{spin}(X)$, nor is it clear that they give the same pairing. We shall establish all these facts.

\subsection{Preliminary Remarks}
Suppose that $W$ is a compact, oriented $4$-manifold with boundary $M$ and suppose that $F\colon (W,M)\to( \Cee P^2,\Cee P^1)$ is a continuous map.
Fix a $2$-sphere $S\subset \Cee P^2$ meeting $\Cee P^1$ transversally in a single point $z$. Deforming $F$ by a homotopy of pairs allows us to assume that the map of pairs is transverse to $S$. By this we mean that $F$ is smooth in a neighborhood of the preimage of $S$, that $F|_M$ is transverse to $z\in S^2$, and that $F$ is transverse to $S$. Let $Z\subset M$ be $F^{-1}(z)$ and let $\Sigma\subset W$ be  $F^{-1}(S)$.
Then $\Sigma$ is a compact oriented surface meeting $M$ normally in $Z$, which is its boundary. As before, the normal bundle of $Z\subset M$ is trivialized. The group $H_4(\Cee P^2,\Cee P^1;\Zee)$ is identified with $\Zee$ with generator the image of the fundamental class of $\Cee P^2$ associated with the complex orientation. 
The map $F_*\colon H_4(W,M;\Zee)\to H_4(\Cee P^2,\Cee P^1;\Zee)$ sends the relative fundamental class of $(W,M)$ to a multiple of this
generator. This multiple is the {\em relative degree} of $F$.
It is also the number of points counted with sign in the preimage under $F$ of a generic point in $\Cee P^2\setminus \Cee P^1$.

\begin{claim}\label{claim1}
Suppose that $U_1$ and $U_2$ are singular integral $2$-cocycles on $\Cee P^2$, each representing the cohomology class Poincar\'e dual to $\Cee P^1$, with the property that the cocycle $(U_1\cdot U_2)|_{\Cee P^1}$ vanishes.
Then $\int_WF^*U_1\cdot F^*U_2$ is the relative degree of $F$.
\end{claim}

 \begin{proof}
$ U_1\cdot U_2$ is a relative cocycle on $(\Cee P^2,\Cee 
P^1)$ and its cohomology class evaluates $1$ on the fundamental cycle of  $\Cee P^2$. Thus, its pullback to $(W,M)$ evaluated on a relative fundamental class of $(W,M)$ is the relative degree of the map $F$.
 \end{proof}
 
  \begin{rem}
Claim~\ref{claim1} holds for ordered simplicial cocycles $U_1, U_2$ if $F$ is an ordered simplicial map.
 \end{rem}
 
\begin{defn} Let $F\colon (W,M)\to (\Cee P^2,\Cee P^1)$ be transverse to $S$ with preimage $\Sigma$ meeting $M$ in $Z$. We define the {\em relative self-intersection of $\Sigma$} to be the intersection number of $\Sigma$ with a  section $\Sigma'$ of the normal bundle of $\Sigma\subset W$, a section
 that is generic in the family of all smooth sections  constant over $Z$ with respect to the given trivialization of the normal bundle of $Z\subset M$.
 \end{defn}

 \begin{claim}\label{cor1}
 Fix $F\colon (W,M)\to (\Cee P^2,\Cee P^1)$. Let $F'$ be a homotopic map of pairs that is smooth and transverse to $S$ with $\Sigma=(F')^{-1}(S)$.
 The relative degree of $F$ is equal to the relative self-intersection of $\Sigma$
 \end{claim}
 
 \begin{proof}
Fix a point $x\in S$ which is a generic value for $F'|_\Sigma$. Then there is a small $4$-ball $B$ centered at
$x$ such that the restriction of $F'$ to $(F')^{-1}(B)=\coprod_i B_i$ is a (trivial) covering projection. We associate to each $B_i$  the sign that is the local degree of $F$ on $B_i$.
Choose $S'$ to be a generic section of the normal bundle of $S$ vanishing only at $x$.
Let $\Sigma'=(F')^{-1}(S')$. It is a generic section of the normal bundle of $\Sigma$ in $W$, a section that is constant (in the given trivialization) on the boundary. Clearly, $\Sigma\cap\Sigma'$ is contained in $(F')^{-1}(B)$ and has one point, a point of transverse intersection, in each $B_i$. The sign of the intersection at that point is the local degree of $F$ on $B_i$. It follows that $\Sigma\cdot \Sigma'$ is the relative degree of $F'$ which is the relative degree of $F$.  
\end{proof}

Now we assume that $W$ is Spin and relate the relative self-intersection to the Spin bordism class determined by the boundary.

\begin{lem}\label{prop1}
Let $W$ be a compact Spin $4$-manifold and let $F\colon (W,M)\to (\Cee P^2,\Cee P^1)$ be transverse to $S$ with pre-image $\Sigma$ with boundary $Z$. Then the relative self-intersection of $\Sigma$ reduced modulo two is equal to $[Spin(Z)]$.
\end{lem}

\begin{proof}
Denote by $T$ the trivialization of the normal bundle of $Z$ induced by $F$. There is another trivialization of this normal bundle, denoted $T'$,
which extends to a normal framing of the union of the non-closed components of $\Sigma$. 
The difference between  $T$ and $T'$ is a map $Z\to SO(2)$ and it represents an element in $H_1(SO(2))=\Zee$. Since $W$ is Spin every closed component of $\Sigma$ has even self-intersection. Hence, the difference  element between $T$ and $T'$ is equal modulo $2$ to the relative self-intersection of $\Sigma$ (with respect to the trivialization $T$). On the other hand, the Spin structure of $Z$ induced by the trivialization $T'$ bounds the Spin manifold consisting of the non-closed components of $\Sigma$ with Spin structure induced by the extension of $T'$. Hence, the Spin bordism class of $Z$ determined by the normal trivialization $T'$ is trivial.  Changing the normal trivialization by a map of $Z\to SO(2)$ changes the Spin bordism class of  $Z$ by the class represented by $Z\to SO(2)\subset SO(3)$ in $H_1(SO(3);\Zee)=\Zee/2$. 
Thus, the Spin bordism class of $Z$ determined by $T$ is equal to the relative self-intersection of $\Sigma$ modulo $2$.
\end{proof}

\begin{cor}\label{cor7}
1. Under the hypotheses of the Lemma~\ref{prop1}, suppose that $U$ is a fundamental singular two-dimensional cocycle for $\Cee P^2$ whose restriction to $\Cee P^1$, denoted $u$, satisfies $u^2=0$. Then $F^*U$ is a $2$-cocycle extending $(F|_M)^*u$ and
$\int_WF^*U^2$ is congruent modulo $2$ to the Spin bordism class of $Z=(F|_M)^{-1}(z)$ for generic $z\in \Cee P^1$. 

2. If $(w,p,0)$ represents an element of $G^1(X)$ and $p=\varphi^*u$ where $u$ is a reduced singular cocycle on $\Cee P^1$ representing the generator and $\varphi$ is an ordered simplicial map, then
$$\langle (w,p,0),(M,f)\rangle_1=\langle (w,p,0),(M,f)\rangle_2.$$
\end{cor}

\begin{proof}
Statement 1 is  immediate from Lemma~\ref{prop1}, Claim~\ref{cor1}, and Claim~\ref{claim1}. It holds both in the singular context and in the ordered simplicial context.
As for the second statement, let $\varphi\colon M\to \Cee P^1$ be an ordered simplicial map and suppose that $p=\varphi^*u$ for a reduced singular cocycle $u$ representing the generator in cohomology.
Then we can extend $\varphi$ to a map $F\colon W\to \Cee P^2$ with $W$ being a compact $4$-dimensional Spin manifold. We can extend the ordered triangulations of $M$ and $\Cee P^1$ to ordered triangulations of $W$ and $\Cee P^2$ in such a way that $F$ can be assumed to be an ordered simplicial map. Let $U$ be a reduced singular cocycle on $\Cee P^2$ extending $u$. 
We now replace $u$ and $U$ by the reduced ordered simplicial cocycles they induce (keeping the same names for these new objects).
Since $u$ is a reduced ordered simplicial cocycle, we see that $u^2=0$. Thus, the first statement of the corollary applies to show that
$\int_WF^*U^2=[Spin(Z)]$ where $Z$ is the preimage of a barycenter of a $2$-simplex of $\Cee P^1$.  Thus, the two definitions agree in this special case.
\end{proof}

\subsection{The pairing $G^1(X)\times \Omega_3^{spin}(X)\to \Ar/\Zee$}
Now we establish that both definitions give the same pairing
$$G^1(X)\times \Omega_3^{spin}(X)\to \Ar/\Zee.$$

\begin{claim}
Fixing $(w,p,0)$, the representative $f\colon M\to X$, and an ordered triangulation of $M$, the value $\langle (w,p,0),(M,f)\rangle_2$ is independent of the choice of Spin manifold $W$ bounding $M$ and of
the cocycle $P$ extending $f^*p$. 
\end{claim}

\begin{proof}
Suppose we have two choices of extensions $(W,P)$ and $(W',P')$. Then the union of $W$ and $-W'$ along $M$ is a closed Spin $4$ manifold and there is a cocycle $Q$ on the union that restricts to $P$ on $W$ and $P'$ on $W'$. Hence,
$$\int_WP^2-\int_{W'}(P')^2=\int_{W\cup -W'}Q^2,$$
and since $W\cup -W'$ is Spin the right-hand side is even. 
The result follows immediately.
\end{proof}

\begin{lem}
Given a Spin bordism representative $f\colon M\to X$ and an ordered triangulation of $M$
and triples $(w,p,0)$ and $(w',p',0)$ that are equal in $G^1(X)$ we have
$$\langle (w,p,0),(M,f)\rangle_2= \langle (w',p',0),(M,f)\rangle_2$$
\end{lem}

\begin{proof}
Choose a Spin $4$-manifold $W$ bounding $M$ over which the cohomology class of $f^*p$ extends. Let $P$ be a cocycle on $W$ extending $p$.
By the definition of equivalence in $G^1(X)$ there is a $1$-cochain $t$ on $X$ such that $p'=p+dt$ and $w'=w+(1/2)(tdt+p\cup_1dt)$.
There is an extension $T$ over $W$ of $f^*t$. Then
\begin{align*} 
\langle (w',p',0),(M,f)\rangle_2 &=\frac{1}{2}\int_W(P+dT)^2+\int_Mf^*w'\\
 &=\frac{1}{2}\int_W(P+dT)^2+\int_Mf^*\bigl(w+(1/2)(tdt+p\cup_1dt)\bigr)\\
 &=\frac{1}{2}\int_WP^2+\int_Mf^*w =\langle(w,p,0),(M,f)\rangle_2,
\end{align*}
where in the last step we use the fact that
$$(PdT+dTP+dT^2)|_M=d\left((1/2)tdt+p\cup_1dt\right)$$
and invoke  Stokes' Theorem.
\end{proof}

\begin{lem}
Fix an ordered triangulation of $M$.
If $f\colon M\to X$ extends to $F\colon W\to X$ where $W$ is a compact Spin $4$-manifold bounding $M$, then for any $(w,p,0)$ representing an element in
$G^1(X)$ we have
$$\langle (w,p,0),(M,f)\rangle_2=0.$$
\end{lem}

\begin{proof}
The cocycles $F^*p$ and $F^*w$ extend $f^*p$ and $f^*w$ respectively, and $dw+(1/2)p^2=0$, hence
\begin{align*}
\langle (w,p,0),(M,f)\rangle_2 &=\frac{1}{2}\int_WF^*p^2+\int_Mf^*w\\
&=\int_W\bigl(\frac{1}{2}F^*p^2+dF^*w\bigr)=0\\
\end{align*}
\end{proof}

Now that we have established bordism invariance, it follows immediately that the pairing given by Definition~\ref{defn2} is independent of the ordered triangulation of $M$ since any two such are connected by an ordered triangulation of $M\times I$.

\begin{cor}\label{rho1cor}
Both Definition~\ref{defn1} and Definition~\ref{defn2} determine the same well-defined pairing
$$G^1(X)\times \Omega_3^{spin}(X)\to \Ar/\Zee.$$
This pairing is linear in the second variable.
\end{cor}

\begin{proof}
The above establishes that the pairing given by Definition~\ref{defn2} is well-defined. It is clearly linear in the second variable. Having established that Defintion~\ref{defn2} determines a well-defined pairing, it follows immediately from Part 2 of Corollary~\ref{cor7} that Definition~\ref{defn1} agrees with Definition~\ref{defn2}.
\end{proof}

The pairing
$$G^1(X)\times \Omega_3^{spin}(X)\to \Ar/\Zee,$$
given in Corollary~\ref{rho1cor} is denoted $\langle (w,p,0),(M,f)\rangle$.
It is natural in $X$ and additive in the second variable and its adjoint is denoted
$$\rho^1\colon G^1(X)\to \DX.$$

\begin{claim}
The pairing produced by Corollary~\ref{rho1cor}  is additive in
$G^1(X)$, or equivalently, its adjoint  $\rho^1$ is a group homomorphism. 
\end{claim}

\begin{proof}
We work with Definition~\ref{defn2}.
 Since $\Omega_3^{spin}(\Cee P^2\times \Cee P^2)=0$, given two cocycles $p,q$ on $M$ there is a Spin $4$-manifold $W$ bounding $M$ over which both cocycles extend, say to $P$ and $Q$. Then
$$\int_W(P+Q)^2= \int_WP^2+\int_WQ^2+\int_Wd(P\cup_1 Q).$$  From this, Stokes' Theorem, and the formula for addition in $G^1(X)$, which is
$$(w,p,0)(v,q,0)=(w+v+p\cup_1 q,p+q,0),$$
the $G^1(X)$ additivity of the pairing is clear.
\end{proof}

Let us consider continuity of the pairing, or equivalently, of $\rho^1$. 

\begin{claim}
The pairing given by Corollary~\ref{rho1cor} is continuous and its adjoint
$$\rho^1\colon G^1(X)\to \DX$$ 
is continuous when $\DX$ is given the Pontrjagin dual  topology.
\end{claim}

\begin{proof}
Since $\Omega_3^{spin}(X)$ is discrete, we need only consider continuity of the pairing in the first variable. Thus, we can fix $f\colon M\to X$. We work with reduced singular cochains.
Fix an element $(w,p,0)$ and  an ordered triangulation of $M$.
There is a neighborhood ${\mathcal U}$ of $(w,p,0)$ in ${\bf C}(X)$ such that for every $(w',p',0)\in {\mathcal U}$ the image of $f^*p'$ in reduced simplicial cocycles on $M$ is equal to the image of $f^*p$.
Choose a Spin $4$-manifold $W$ bounding $M$ over which the cohomology class of $p$ extends. Extend the ordered triangulation to $W$. Let $P$ be a reduced singular cocycle extending $f^*p$. Then for each $(w',p',0)\in{\mathcal U}$ we can choose a reduced singular cocycle $P'$ extending $f^*p'$ so that the image of $P'$ in reduced simplicial cocycles on $W$ is equal to that of $P$. It follows that
$$\int_W(P')^2=\int_WP^2,$$
and hence 
$$\langle(w',p',0),(M,f) =\frac{1}{2}\int_WP^2+\int_Mw'$$
which is a continuous function of $w'$.

Since the Pontrjagin dual topology is the compact-open topology the continuity of the pairing implies that  its adjoint 
$$\rho^1\colon G^1(X)\to {\rm Hom}(\Omega_3^{spin}(X),\Ar/\Zee)$$
is continuous.
\end{proof}

\subsection{Injectivity and Image of $\rho^1$}

Now we  show that $\rho^1$ is  injective and identify its image.
Recall from Section~\ref{sect2.1} that there is an increasing filtration $F_*(X)$ on $\Omega_3^{Spin}(X)$ where $F_i(X)$ is the image of
$\Omega_3^{Spin}(X^{(i)})\to \Omega_3^{Spin}(X)$, where $X^{(i)}$ is the $i$-skeleton of some CW decomposition of $X$. There is the dual filtration
$${\rm Hom}(\Omega_3^{spin}(X),\Ar/\Zee)\supset F^1(X)\supset F^2(X)\supset F^3(X)=0,$$
where $F^i(X)$ is the subgroup of homomorphisms vanishing on $F_i(X)$ for $1\le i\le 3$.

\begin{prop}\label{filterprop}
\begin{enumerate}
\item [(a)] The continuous homomorphism 
$$\rho^1\colon G^1(X)\to {\rm Hom}(\Omega_3^{Spin}(X),\Ar/\Zee)$$
is injective and its image is
$$F^1(X)={\rm Hom}(\Omega_3^{spin}(X)/F_1(X),\Ar/\Zee),$$
and hence  $\rho^1\colon G^1(X)\to F^1(X)$ is an isomorphism of compact abelian groups.
\item[(b)] The restriction of $\rho^1$ to $G^2(X)$ defines a homomorphism
$$G^2(X)\to  {\rm Hom}(\Omega_3^{Spin}(X)/F_2(X),\Ar/\Zee).$$
With the identifications of $G^2(X)$ with $H^3(X;\Ar/\Zee)$ given in Proposition~\ref{4.1} and of $\Omega_3^{Spin}(X)/F_2(X)$ with $H_3(X;\Zee)$ given in \S\ref{sect2.1}, this map is the natural homological pairing and hence is an isomorphism. 
\item[(c)] Under the identifications of $G^1(X)/G^2(X)$ with $SH^2(X;\Zee/2)$ given in Proposition~\ref{4.2} and of $F_2(X)/F_1(X)$ with $SH_2(X;\Zee/2)$ given in \S\ref{sect2.1} 
the pairing induced by $\rho^1$
$$G^1(X)/G^2(X)\to {\rm Hom}(F_2(X)/F_1(X),\Ar/\Zee)$$
 is the natural homological pairing and hence is an isomorphism.
\end{enumerate}
\end{prop}

\begin{proof}
Let $X$ be an ordered simplicial complex. We can assume any Spin bordism representative is  given by
an ordered simplicial map from an ordered triangulation of a Spin $3$-manifold.
We claim that the homomorphism defined by any $(w,p,0)$ vanishes on any Spin bordism element of the form
$f\colon M\to X^{(1)}\subset X$. The reason is that every element in $G^1(X)$ has a representative $(w,p,0)$ with $w$ and $p$ reduced singular cochains
of degrees $3$ and $2$, respectively.
Thus, if the image $f(M)$ lies in the $1$-skeleton,  then $f^*p$ and $f^*w$ vanish    and hence the evaluation of $(w,p,0)$ on $f\colon M\to X$ vanishes as well. This proves that $\rho^1(G^1(X))$ is contained in $F^1(X)$.

The subgroup $G^2(X)$ is identified with $H^3(X;\Ar/\Zee)$ by sending $(w,0,0)\mapsto [w]$ (note that for such elements $w$ is closed).
The evaluation of such an element on a Spin bordism class $f\colon M\to X$ is given by
$\int_Mf^*w.$ Clearly, then this map vanishes if $f(M)$ is contained in the $2$-skeleton of $X$.
Thus, $\rho^1$ induces a map
$$G^2(X)\to {\rm Hom}(\Omega_3^{spin}(X)/F_2(X)),\Ar/\Zee) =F^2(X).$$
The identification of the quotient $\Omega_3^{spin}(X)/F_2(X)$  with $H_3(X;\Zee)$ sends the class of 
$f\colon M\to X$ to $f_*([M])$.  Thus, under these identifications $G^2(X)=H^3(X;\Ar/\Zee)$ and $\Omega_3^{spin}(X)/F^2(X)=H_3(X;\Zee)$. The evaluation map is identified with the natural homological evaluation map
$$H^3(X;\Ar/\Zee)\to {\rm Hom}(H_3(X;\Zee),\Ar/\Zee),$$
which is an isomorphism by Pontrjagin duality. This proves Part (b).

It follows that $\rho^1$ induces a
 continuous homomorphism
 $$G^1(X)/G^2(X)\to F^1(X)/F^2(X).$$

Recall that 
$$F_2(X)/F_1(X)=SH_2(X;\Zee/2)=H_2(X;\Zee/2)/s(H_4(X;\Zee))$$
where the map $s$ is the composition 
$$H_4(X;\Zee)\to H_4(X;\Zee/2)\buildrel (Sq^2)^*\over \longrightarrow H_2(X;\Zee/2).$$
The quotient $G^1(X)/G^2(X)$ is identified with $SH^2(X;\Zee/2)$ by the map
$$(w,p,0)\mapsto [p]\in H^2(X;\Zee/2)$$
(since $dw+(1/2)p^2=0$, the cohomology class $[p]$ lies in $SH^2(X;\Zee/2)$.)

We claim that with these identifications the pairing
$$\bigl(G^1(X)/G^2(X)\bigr)\times \bigl(F_2(X)/F_1(X)\bigr)\to \Ar/\Zee$$
is the Pontrjagin pairing between homology and cohomology and hence its adjoint is an isomorphism of topological abelian groups.
Since the restriction map $G^1(X)/G^2(X)\to G^1(X^{(2)})/G^2(X^{(2)})$ is an injection, it sufficies to prove this result in the case when $X$ is
$2$-dimensional. In this case $X/X^{(1)}$ is a wedge of $2$-spheres and since the  map induced on $G^1$ by the projection $X\to X/X^{(1)}$ is surjective it suffices to prove the result for
$X$ a wedge of $2$-spheres. By additivity, it suffices to prove the result for $S^2=\partial \Delta^3$. 
Let $u$ be a reduced singular cocycle on $S^2$ whose cohomology class is the generator and let
$\varphi\colon M\to S^2$ represent a Spin bordism element. We can assume that $\varphi$ is an ordered simplicial map. Thus, according to Definition~\ref{defn1} the value of $(0,u,0)$ on $(M,\varphi)$ is the Spin bordism class represented by $\varphi$. Thus, the evalaution of $(0,u,0)$ is
the isomorphism between $\Omega_3^{spin}(S^2)$ and $\Omega_1^{spin}=\Zee/2$, showing that for $S^2$ the pairing is identified with the usual homological pairing.

This shows that the pairing $G^1(X)/G^2(X)\to F^1(X)/F^2(X)$ is the identity under the identifications of each of these groups $SH^2(X;\Zee/2)$.
This  proves Part (c).

This completes the proof that on the associated gradeds the induced maps are the identities under the identifications of the associated gradeds with (subgroups of) cohomology groups. Hence, $\rho^1\colon G^1(X)\to F^1(X)$ is an isomorphism of compact abelian groups. This proves Part (a) and completes the proof of the proposition. 
\end{proof}

\section{The Continuous Function $\lambda^1$}

\subsection{The Arf Invariant}\label{7.1}
The function $\lambda^1\colon H^1(X;\Zee/2)\to {\rm Hom}(\Omega_3^{Spin}(X),\Ar/\Zee)$ is based on the Arf invariant of certain quadratic forms.
Let $M$ be a closed $3$-dimensional Spin manifold and let $\alpha$ be an element of $H^1(M;\Zee/2)$. The cohomology class of $\alpha$ is equivalent to  a map $\alpha\colon M\to \Ar P^\infty$ well-defined up to homotopy.
Deform $\alpha\colon M\to \Ar P^\infty$ until it is transverse to a codimension-$1$ sub projective space. The preimage, denoted $\Sigma_\alpha$, is a surface in $M$. Let $\eta$ denote its normal bundle.
The Spin structure on $M$ determines a trivialization of $TM|_{\Sigma_\alpha}$ up to homotopy and hence a trivialization of $T\Sigma_\alpha\oplus \eta$ up to homotopy.
This trivialization allows us to define a quadratic function
$$q_\alpha\colon H^1(\Sigma_\alpha;\Zee/2)\to \frac{1}{2}\Zee/2\Zee.$$
 Let $\gamma$ be an embedded loop in $\Sigma_\alpha$. The value of $q_\alpha$ on the Poincar\'e dual $[\gamma]^*$ to the fundamental class of $\gamma$ is defined as follows. Fix a trivialization $T$ of the normal bundle of $\gamma\subset M$ with the property that the induced 
Spin structure on $\gamma$ is the trivial one. Then $q_\alpha([\gamma]^*)$ is defined to be the number of full right-hand twists with respect to $T$ of a normal  tube about $\gamma$ in $\Sigma_\alpha$.
This is an integer if $\Sigma_\alpha$ is orientable in a neighborhood of $\gamma$. Otherwise, it is a half-integer. Direct arguments show that this gives a well-defined quadratic function on $H^1(\Sigma_\alpha;\Zee/2)$, quadratic in the sense that
$$q_\alpha(x+y)=q_\alpha(x)+q_\alpha(y)+x\cdot y.$$
This quadratic function has an Arf invariant ${\rm Arf}(q_\alpha)$ in the $8^{th}$-roots of unity given by the formula
 $${\rm Arf}(q_\alpha) = \frac{1}{\sqrt{|H^1(\Sigma_\alpha;\Zee/2)|}}\sum_{x\in H^1(\Sigma_\alpha,\Zee/2)}e^{i\pi q_{\alpha}(x)}.$$
 A standard bordism argument shows that this Arf invariant only depends on $M$ and $\alpha$. We denote it ${\rm Arf}(M,\alpha)$.
 
 \begin{rem}
(a). There is another way to view this quadratic form which goes back to Brown \cite{Brown}, following a suggestion of Dennis Sullivan and related prior work of Tony Phillips on bordism of immersed surfaces. The restriction of the trivialization of $TM$ to a regular neighborhood $N_\gamma$ of $\gamma$ in $M$ induces an immersion of $N_\gamma$ into $\Ar^3$ well defined up to regular homotopy. We can assume that the immersion in fact is an embedding with the image of $\gamma$ being the unit circle. Then we count the number of right-hand twists of $\Sigma_\alpha\cap N_\gamma$ in $\Ar^3$.

(b). One can define a $Pin^-$ surface to be a surface $\Sigma$ together with a trivialization up to homotopy of 
$T\Sigma\oplus \eta$ for some line bundle $\eta$. Clearly, then $\Sigma_\alpha$ is a $Pin^-$ surface.
The same definition of the quadratic function on the first cohomology is valid for $Pin^-$-surfaces and
the Arf invariant of a $Pin^-$ surface determines an isomorphism between the $2$-dimensional $Pin^-$ bordism group and $\Zee/8$.
 \end{rem}

For $\alpha\in H^1(X;\Zee/2)$ we define the value of $\alpha$ on a Spin bordism representative $f\colon M\to X$ to be ${\rm Arf}(M,f^*\alpha)$, viewed  as an element of $\Ar/\Zee$ using the identification $\Ar/\Zee \to S^1$ given by $t\mapsto {\rm exp}(2\pi it)$.  If $f\colon M\to X$ is the Spin boundary of $F\colon W\to X$, then the surface $\Sigma_{f^*\alpha}$ bounds an embedded
 $3$-manifold $N\subset W$. Then the Spin structure on $W$ induces a $Pin^-$ structure on $N$, showing that the $Pin^-$ bordism class of $\Sigma_{f^*\alpha}$ 
 is zero, and hence that its ${\rm Arf}$ invariant vanishes. A more direct   way to think about this is that the quadratic form on $H^1(\Sigma_{f^*\alpha};\Zee/2)$ is trivial on any $1$-dimensional class in $\Sigma_{f^*\alpha}$ that extends over $N$. This produces a lagrangian subspace of $H^1(\Sigma_{f^*\alpha};\Zee/2)$ on which the quadratic form vanishes, and hence the Arf invariant of the quadratic form of $\Sigma_{f^*\alpha}$ is $0$ in $\Ar/\Zee$.
 
 The evaluation of $\alpha$ on Spin bordism representatives thus determines a function on Spin bordism classes
  $\Omega_3^{spin}(X)\to \Ar/\Zee$,
which is clearly linear and whose image lies the subgroup of order $8$.
 Its adjoint, as $\alpha$ varies, is a continuous function (but not a homomorphism) 
 $$ \lambda^1\colon H^1(X;\Zee/2)\to {\rm Hom}(\Omega_3^{spin}(X),\Ar/\Zee).$$
 It is a natural transformation between functors from the homotopy category to the category of compact topological spaces and continuous maps. 
 
 \begin{claim}\label{claim2.1}
 Under the natural identification $F_1(X)= H_1(X;\Zee/2)$ given in \S\ref{sect2.1}, the composition 
$$H^1(X;\Zee/2)\buildrel\lambda^1\over\longrightarrow {\rm Hom}(\Omega_3^{spin}(X),\Ar/\Zee)\to {\rm Hom}(F_1(X),\Ar/\Zee)$$
 is the natural continuous isomorphism
 $$H^1(X;\Zee/2)\to {\rm Hom}(H_1(X;\Zee/2),\Ar/\Zee)$$
 given by the adjoint of evaluation of cohomology class on homology classes.
 \end{claim}
 
 \begin{proof}
We can assume that $X$ is an ordered complex.
 An element of $F_1(X)$ is represented by an ordered simplicial map $f\colon M\to X^{(1)}$. For each $1$-simplex $\tau$ choose a point $x_\tau$ in the interior of $\tau$ and let $\Sigma_\tau$ be the pre-image of $x_\tau$. It inherits a Spin structure from that of $M$ and an orientation of $\tau$.  Denote by $[\Sigma_\tau]$ the resulting element in $\Omega_2^{spin}(pt)=\Zee/2$. The element
$$ \sum_\tau \tau\otimes [\Sigma_\tau]$$
is an ordered simplicial one-chain with $\Zee/2$ coefficients. It is a cycle and its homology class gives the identification of $F_1(X)$ with $H_1(X;\Zee/2)$.

Given a $1$-cocycle $a$, the surface $\coprod_\tau\langle a,\tau\rangle \Sigma_\tau$ is Poincar\'e dual to the cohomology class of $a$. Hence, the
value of $\lambda^1([a])$ on  $(M,f)$  is 
$$\sum_\tau\langle a,\tau\rangle {\rm Arf}(\Sigma_\tau).$$
Since $\Sigma_\tau$ is a Spin surface, ${\rm Arf}(\Sigma_\tau)$  is equal to $(1/2)[\Sigma_\tau]$, (i.e., to one-half the Spin borism class of $\Sigma_\tau$).
Thus, the value of $\lambda^1([a])$ on $(M,f)$ is
$$\frac{1}{2}\sum_\tau\langle a,\tau\rangle[\Sigma_\tau],$$
which is exactly the evaluation  of the cohomology class $[a]$ on the class in $H_1(X;\Zee/2)$ represented by  $(M,f)$ under the above identification of $F_1(X)$ with $H_1(X;\Zee/2)$ (together with the embedding $\Zee/2\subset \Ar/\Zee$ given by multiplication by $1/2$).

 \end{proof}

\subsection{An Example} 
 The tangent circle bundle of $S^2$, denoted  $T_SS^2$, is the boundary of of the tangent disk bundle $T_{D^2}S^2$. The latter has a natural orientation and consequently, because it is simply connected, there is a natural Spin structure. By the standard Spin structure on 
$T_SS^2$ we mean the one induced as the boundary of  $T_{D^2}S^2$. [Our convention is to orient the boundary  by taking
the outward pointing normal as the first vector.] 

\begin{claim}\label{RP3}
Let $\alpha \in H^1(T_SS^2;\Zee/2)$ be the non-zero class. Then
$${\rm Arf}(T_SS^2,\alpha)=1/8.$$
\end{claim}

\begin{proof}
In \cite{Guillou, KirbyTaylor, Matsumoto} it is shown that if $M$ is a closed Spin $3$-manifold and $\alpha\in H^1(M;\Zee/2)$ then  
\begin{equation}\label{KirbyTaylor}{\rm Arf}(M,\alpha)=\bigl(\sigma(W)-\sigma(W_\alpha)\bigr)/16\in \Ar/\Zee.
\end{equation}
Here $\sigma$ is the signature, $W$ is a Spin $4$-manifold with boundary $M$ and $W_\alpha$ is a Spin $4$-manifold with boundary the Spin manifold obtained from $M$ by twisting the Spin structure by $\alpha$. [Recall that by Rochlin's theorem $\sigma(W)$ and $\sigma(W_a)$ are well-defined modulo $16$.] In case of the claim we have $W=T_{D^2}S^2$ and $\sigma(W)=1$. Let  $W'$ be $W$ with the opposite orientation. Its boundary determines a Spin structure on $T_SS^2$ which is not isomorphic to the original Spin structure
since $\sigma(W')\not= \sigma(W)\pmod {16}$. Thus the boundary of $W'$ is the twist of the given Spin structure on $T_SS^2$ by the non-zero class $\alpha\in H^1(T_SS^2;\Zee/2)$. That is to say $W'=W_\alpha$.
 Thus, $\sigma(W)-\sigma(W_\alpha)=2$ and ${\rm Arf}(T_SS^2,\alpha)=1/8$.
 \end{proof}

\section{Proof of the Compatibility}\label{sectproof}

Now to the last step: We establish that $\rho^1$ and $\lambda^1$ satisfy the condition given in Proposition~\ref{4.8}.
Let $f\colon M\to X$ represent an element of $\Omega_3^{spin}(X)$. In our situation the required compatibility is the following equation in $\Ar/\Zee$:
\begin{align*}
{\rm Arf}(M_\alpha)+{\rm Arf}(M_\beta)=
 {\rm Arf}(M_{\alpha+\beta})+\langle\bigl(\frac{1}{2}a(a\cup_1b)b+\frac{1}{4}AB^2,ab,0\bigr),(M,f)\rangle,
\end{align*}
where $a$ and $b$ are cocycles on $X$, $\alpha=[f^*a]$, and $\beta=[f^*b]$.
 According to Proposition~\ref{4.8}, if this is established then there is a continuous homomorphism 
 $$\rho\colon G(X)\to {\rm Hom}(\Omega_3^{spin}(X),\Ar/\Zee)$$
  defined by $\rho(w,p,a)=\rho^1(w,p,0)\lambda^1(a)$.
It will then be easy to show that 
$\rho$ is an isomorphism.

\subsection{Non-additivity of the Arf Invariant}

Given a space $X$ and cocycles $a,b\in Z^1(X;\Zee/2)$  define a function
$$\varphi(a,b)\in {\rm Hom}(\Omega^{spin}_3(X),\Ar/\Zee)$$
by 
$$\varphi(a,b)(M,f)={{\rm Arf}(M,f^*a)+{\rm Arf}(M,f^*b)-\rm Arf}(M,f^*(a+b)).$$
We have already observed that  $\varphi(a,b)$ depends only on the cohomology classes of $a$ and $b$, and hence it induces a natural transformation 
$$H^1(X;\Zee/2)\times H^1(X;\Zee/2)\to{\rm Hom}(\Omega^{spin}_3(X),\Ar/\Zee),$$
both functors being considered as functors on the homotopy category to the category of sets.

We shall not make use of the content of the following remark but we feel it helps clarify what is going on.
\begin{rem}\label{8.1}Let $a$ and $b$ be cocycles representing classes
$\alpha,\beta\in H^1(M;\Zee/2)$ and let $\Sigma_\alpha$ and $\Sigma_\beta$ be dual surfaces meeting transversely. Then
$$\varphi(a,b)(M,{\rm Id}))={\rm Arf}(M,\alpha)+{\rm Arf}(M,\beta)-{\rm Arf}(M,\alpha+\beta)=\frac{1}{2}q_\beta(\alpha|_{\Sigma_\beta})\in (1/4)\Zee/\Zee.$$
This can be proved with a delicate geometric argument or directly from Equation~\ref{KirbyTaylor} and the following elementary exercise. Let $M_\alpha$ denote the $3$-manifold $M$ with Spin structure twisted from the given Spin structure by $\alpha$. Then:
\begin{enumerate}
\item The quadratic form $q'$ for $\Sigma_\beta\subset M_\alpha$ is given by
$$q'(x)=q_\beta(x)+\langle x,\alpha\rangle,$$
and consequently
\item
$${\rm Arf}(q')={\rm Arf}(q_\beta)-\frac{1}{2}q_\beta(\alpha|_{\Sigma_\beta}).$$
\end{enumerate}
\end{rem}

Rather than use this description of $\varphi(a,b)$ we study the formal properties of the non-additivity and use these together with one explicit computation to
determine the non-additivity as the evaluation of a particular element of $G^1(X)$.

\begin{claim}\label{varphiclaim}
\begin{enumerate}
\item $\varphi(a,b)=0$ if either $a$ is exact or $b$ is exact. 
\item Let $c\in Z^1(T_SS^2;\Zee/2)$ represent the generator in cohomology. Then
$$\varphi(c,c)(T_SS^2,{\rm Id})=1/4.$$
\item If $f\colon M\to X^{(1)}$, then for any $a,b\in Z^1(X;\Zee/2)$ we have
$$\varphi(a,b)(M,f)=0.$$
\end{enumerate}
\end{claim}

\begin{proof}
Item (1) is immediate from the definition. Item (2), follows from Claim~\ref{RP3} which says that ${\rm Arf}(T_SS^2,c)=1/8$, and the fact that since $[2c]=0$, ${\rm Arf}(T_SS^2,2c)=0$.
For Item (3), the surfaces $\Sigma_a$ and $\Sigma_b$ will be transverse pre-images of points in the interior of the $1$-cells of $X^{(1)}$, and hence can be taken to be disjoint. Thus, we can take $\Sigma_{a+b}=\Sigma_a\coprod \Sigma_b$, and Item (3) is clear.
\end{proof}

From Items (1) and (3) of Claim~\ref{varphiclaim}  we conclude:

\begin{cor}
 $\varphi$ induces a natural transformation
 $$H^1(X;\Zee/2)\times H^1(X;\Zee/2) \to G^1(X),$$
 where each of these functors is considered as a functor from the homotopy category to the category of sets.
\end{cor}


\begin{claim}
For $a,b\in Z^1(X;\Zee/2)$,
the image of $\varphi(a,b)$ under the projection $G^1(X)\to SH^2(X;\Zee/2)$ is the cohomology class of
$ab$.
\end{claim}

\begin{proof}
The image of $\varphi(a,b)$ in $SH^2(X;\Zee/2)$ is natural in $X$ and the classes $a,b$. Thus, it suffices to prove the result for $X=K(\Zee/2,1)\times K(\Zee/2,1)$ with $a,b$ representing the generators of the first cohomology groups of the factors. In this case $SH^2(X;\Zee/2)$ has a basis consisting of the cohomology classes of $a^2,ab,b^2$.
By Item (1) in Claim~\ref{varphiclaim} we see that the image of $\varphi(a,b)$ is either the class of $ab$ or is $0$.

By Item (2) of Claim~\ref{varphiclaim}, $\varphi(c,c)$ is an element of order $4$ in $G^1(T_SS^2)$. Since $G^2(T_SS^2)=\Zee/2$ if follows that
the image of $\varphi(c,c)$ in $SH^2(T_SS^2;\Zee/2)$ is non-trivial. Consequently, by naturality $\varphi(a,b)$ is non-zero, and hence is equal to the class of  $ab$.
\end{proof}

This is enough to determine $\varphi(a,b)$ up to two possibilities.

\begin{claim} For $a,b\in Z^1(X;\Zee/2)$ we have
$$\varphi(a,b)=(\frac{1}{2}a(a\cup_1 b)b\pm \frac{1}{4}AB^2,ab,0)\in G^1(X).$$
The correct sign is universal, that is to say, independent of $X$, $a$, and $b$.
\end{claim}

\begin{proof}
Consider the natural transformation that associates to $a,b\in Z^1(X;\Zee/2)$ the element
$$\xi(a,b)=(\frac{1}{2}a(a\cup_1 b)b+\frac{1}{4}AB^2,ab,0)\in G^1(X).$$

By Propositions~\ref{4.4} and~\ref{4.5}, $\xi(a,b)$ depends only on the cohomology classes of $a$ and $b$. Clearly, $\xi(a,b)$ is contained in $G^1(X)$ and under the
map $G^1(X)\to SH^2(X;\Zee/2)$ $\xi(a,b)$ maps to the cohomology class of $ab$. 
Thus, the difference $\varphi-\xi$ induces a natural transformation 
$$H^1(X;\Zee/2)\times H^1(X;\Zee/2)\to G^2(X)=H^3(X;\Ar/\Zee)$$
vanishing when either $[a]=0$ or $[b]=0$.
By naturality, it suffices to consider $X=K(\Zee/2),1)\times K(\Zee/2,1)$. In this case $H^3(X;\Ar/\Zee)\equiv (\Zee/2)^3$ with $\Zee/2$-basis being the cohomology classes of $(1/2)[a^3],(1/2)[b^3],(1/2)[ab^2]=(1/2)[a^2b]$, [Since $Sq^1([ab])=[a^2b]+[ab^2]$, the elements $(1/2)[ab^2]$ and $(1/2)[a^2b]$ are equal.]
It follows that the difference $\varphi(a,b)-\xi(a,b)$ is either $0$ or $(1/2)[ab^2]$.
These lead to the two possibilities for $\varphi(a,b)$ listed in the claim. Naturality implies that the sign is universal.
\end{proof}

To determine $\varphi(a,b)$ we will make an explicit computation for $T_SS^2$.

\subsection{A Result for $T_SS^2=\Ar P^3$.}

\begin{prop}\label{3.6}
Equip $T_{S}S^2$ with the Spin structure it inherits as the boundary of the Spin manifold $T_{D}S^2$
associated with its natural orientation.
 Let $c\in Z^1(T_{S}S^2;\Zee/2)$ represent the non-trivial cohomology class. Then
  the class in $G^1(T_{S}S^2)$ represented by 
$$((1/2)c(c\cup_1c)c+(1/4)C^3,c^2,0)$$
evaluates to $1/4$ on the identity map considered as an element of $\Omega_3^{spin}(T_SS^2)$.
\end{prop}

This result is proved by direct combinatorial computation in Section~\ref{A.3}. Notice $c^3$ evaluates to
$1$ on $[T_SS^2]$  so that changing the sign of the 
$(1/4)C^3$ term changes its value on the identity map.

\begin{cor}\label{8.7}
For $a,b\in Z^1(X;\Zee/2)$ we have
$$\varphi(a,b)=(\frac{1}{2}a(a\cup_1b)b+\frac{1}{4}AB^2,ab,0)\in G^1(X).$$
\end{cor}

\begin{proof} (Assuming Proposition~\ref{3.6})
According to Claim~\ref{varphiclaim} if $c\in H^1(T_SS^2;\Zee/2)$ is the non-trivial element we have
$\varphi(c,c)=1/4$. According to Proposition~\ref{3.6} the class
$\bigl((1/2)c(c\cup_1c)c+(1/4)C^3,c^2,0\bigr)$ also evaluates to $1/4$ on the identity map of $T_SS^2$ to itself.
Thus,  
$$\varphi(c,c)=\bigl((1/2)c(c\cup_1c)c+(1/4)C^3,c^2,0\bigr).$$
Of the two possibilities for the formula for $\varphi(a,b)$ only the stated one is consistent with this example.
\end{proof}

Assuming Proposition~\ref{3.6}, this completes the proof of the required compatibility.

\subsection{Isomorphism}

Since $\rho^1$ and $\lambda^1$ satisfy the condtion in Propostion~\ref{4.8} the function the resulting function $$\rho\colon G(X)\to {\rm Hom}(\Omega_3^{spin}(X),\Ar/\Zee)$$
defined by $\rho(w,p,a)=\rho^1(w,p,0)\lambda^1(a)$
is a continuous homomorphism. In Proposition~\ref{filterprop}, we established that
$$\rho^1\colon{G^1(X)}\to {\rm Hom}(\Omega^{spin}_3(X),\Ar/\Zee)$$ is injective with image $F^1(X)$.
Indeed, according to this proposition, $\rho^1$ is a filtered isomorphism in the sense that
 it induces an isomorphism on the associated gradeds
  $$G^2(X)\to F^2(X)$$
  and
  $$G^1(X)/G^2(X)\to F^1(X)/F^2(X),$$
Also, by Claim~\ref{claim2.1} the composition 
$$G(X)/G^1(X)\buildrel\lambda^1 \over\longrightarrow {\rm Hom}(\Omega^{spin}_3(X),\Ar/\Zee)\to {\rm Hom}(F_1(X),\Ar/\Zee)=F(X)/F^1(X)$$
is an isomorphism and indeed under the natural identification of the range with ${\rm Hom}(H_1(X;\Zee/2),\Ar/\Zee)$ it is the adjoint of the homological pairing.
  It follows that $\rho\colon G(X)\to {\rm Hom}(\Omega_3^{spin}(X),\Ar/\Zee)$ is compatible with the filtrations, i.e. $\rho(G^i(X))\subset F^i(X)$,
 and $\rho$ is a filtered group isomorphism. Since the domain and range are compact topological (Hausdorff) groups and $\rho$ is continuous, it follows that $\rho$ is an isomorphism of filtered topological groups, and hence that $\rho$ is a natural transformation between functors from the homotopy category to the category of compact topological abelian groups and continuous homomorphisms. 
  
  Assuming Proposition~\ref{3.6}, this completes the proof of Theorem~\ref{1.2}.
  
\section{Some Results About Spin Structures} 
\subsection{An Action of $H^1(\ZZ/2)$ on Spin Bordism Elements}
Let $b \in Z^1(X, \ZZ/2)$ act on $\Omega_3^{spin}(X)$ by $(M \xrightarrow{f} X)\mapsto (M_{\beta} \xrightarrow {f} X)$, where $\beta = [f^*(b)] \in H^1(M, \ZZ/2)$, and where $M_\beta$ means the  Spin structure on $M$ changed by $\beta$.  The action by $\beta$ defines an order 2 linear automorphism of $\Omega_3^{spin}(X)$. Dualizing gives an order 2 automorphism $\chi = \chi_b: G(X) \to G(X)$ via $\langle\chi(g),\ (M \xrightarrow{f} X) \rangle\ =\ \langle g,\ (M_{\beta} \xrightarrow{f} X) \rangle$.\\

We pose the question, what is the formula for $\chi_b(w,p,a)$?\\
\begin{prop}\label{deltaspin} 
$$\chi_b(w,p,a) =(w,p,a)((1/2)pb + (1/2)a(a\cup_1b)b - (1/4)AB^2,\ ab,\ 0).$$
Multiplying out the right hand side in $G(X)$ gives 
\begin{align*}
\lefteqn{\chi_b (w,p,a) =} \\ 
& &(w+ (1/2)pb + (1/2)p \cup_1 ab + (1/2)a(a\cup_1 b)b - (1/4)AB^2,\ p + ab,\ a).
\end{align*}
\end{prop}

\begin{proof}  Fixing $(w,p,a)$, write  $\chi_b(w,p,a) = (w,p,a)g'$ for some $g' \in G$.  Then we must have $$\langle g', (M \xrightarrow{f} X)\rangle = \langle(w,p,a), (M_{\beta} \xrightarrow{f} X) \rangle - \langle(w,p,a), (M \xrightarrow{f} X) \rangle.$$

Without loss of generality, we can assume $(f^*w,f^*p) \in G^1(M)$ is replaced by $(w_u, p_u)$ where $p_u = z^*(u)$ for some map $z\colon M \to S^2$.  Then
$$\langle(w,p,a), (M \xrightarrow{f} X) \rangle\ =\ {\rm Arf}(M, \Sigma_{\alpha})+ (1/2)[Spin(z, M)] +\int_{[M]} w_u  $$
and $$\langle(w,p,a), (M_{\beta} \xrightarrow{f} X) \rangle\ =\ {\rm Arf}(M_{\beta}, \Sigma_a)  + (1/2)[Spin(z, M_\beta)] + \int_{[M]} w_u.$$

The Spin term changes by $(1/2)\int_{[M]} f^*(pb)  $, which accounts for the $(1/2)pb$ term in our answer.  Namely, interpret $z$ as a framed 1-manifold $Z \times D^2 \subset M$.  Change of Spin structure by $\beta$ means change a framing of the tangent bundle of $M - (point)$ by $M \xrightarrow{\beta} \RR P(3) \to SO(3)$.  Restricting to $Z \times D^2$, the framing changes by 
$$\langle\beta, [Z]\rangle = \langle pb, [M]\rangle .$$.

The Arf term changes by $${\rm Arf}(M_{\beta}, \Sigma_\alpha) - {\rm Arf}(M, \Sigma_\alpha) =  -\frac{1}{2}\ q_\beta(\alpha|_{\Sigma_{\beta}}).$$  Here $\alpha = [f^*(a)] $ and the value of the quadratic form $q$ is interpreted in $\Ar/\ZZ$.  But we know from Corollary~\ref{8.7} and Remark~\ref{8.1} that
$$\frac{1}{2}q_{\beta}(\alpha\Sigma_\beta) = \langle((1/2)a(a\cup_1b)b + (1/4)AB^2, ab, 0),\ (M \xrightarrow{f} X)\rangle \in \Ar/ \ZZ.$$
To change the sign of $q$ just replace $(1/4)AB^2$ by $(-1/4)AB^2$.  Then one sees that $$g' = ((1/2)pb + (1/2)a(a\cup_1b)b - (1/4)AB^2,\ ab,\ 0)$$ and the formula is proved.\end{proof}

\begin{rem}The function $\chi \colon H^1(X, \ZZ/2\ZZ) \to Aut(G(X))$ given above is a group homomorphism.  Both this fact and the fact that $\chi(b) = \chi_b$ is indeed an automorphism of $G(X)$ follow trivially from the interpretations on the Spin bordism side.\end{rem}

\begin{rem}The operation $(M \xrightarrow{f} X)\mapsto (M_{\beta} \xrightarrow {f} X)$ makes sense even if $\beta \in H^1(M, \ZZ/2)$ doesn't come from $X$.  If $(w,p,a) \in G(X)$ then one gets a formula for the difference$$\langle(w,p,a), (M_{\beta} \xrightarrow{f} X) \rangle - \langle(w,p,a), (M  \xrightarrow{f} X) \rangle \in \RR /\ZZ$$  as the evaluation of the element $$([(1/2)(f^*p)b + (1/2)f^*a(f^*a \cup_1 b)b - (1/4)(f^*A)B^2],\  (f^*a)b,\  0) \in G(M)$$ on $M \xrightarrow{Id} M$. Just replace $M \xrightarrow{f} X$ by $M \xrightarrow{Id} M$ and work in $G(M)$ in the proposition.\end{rem}

\subsection{Characterizing Spin Structures on 3-Manifolds}

\begin{prop}Let $M$ be an oriented connected 3-manifold. Choose an ordered simplicial structure $[M]$.   Spin structures on $M$ are in canonical bijective correspondence with functions $Q\colon Z^2([M], \ZZ/2) \to \Ar/\Zee$ that satisfy 
\begin{enumerate}
 \item $Q(dt) = (1/2)\int_{[M]}t dt$ \ \ \ and
\item $Q(p+q) = Q(p) + Q(q) + (1/2)\int_{[M]} p \cup_1 q.$
\end{enumerate}
\end{prop}
\begin{proof} Note that Condition 2 implies that $Q(0)=0$ and, since $p\cup_1p=0$,  $2Q(p)=0$ all $p\in Z^2([M],\Zee/2)$. We exploit the duality between $G_r([M])$ and $\Omega^{spin}_3(M)$.  A Spin structure on $M$ defines a homomorphism  $G = G_r([M]) \to \R / \Z$, by evaluating $(w,p,a)$ on the identity map $M = M$.  Of course we have formulas for this evaluation. But we are more interested in comparing with other  evaluations of $(w,p,a)$ on maps $M' \to M$, where $M'$ is another Spin structure on $M$.  We will only need to consider the evaluations of elements $(w,p,0) \in G^1 = G^1([M])$.  \\

Since $M$ is an oriented connected 3-manifold, the group $G^1=G^1_{r}([M])$ is simple to describe.  All degree 3 cochains are cocycles, and $H^3([M]; \R / \Z) = \R / \Z$.
The group $G^1$ is the quotient  $$\R / \Z \times Z^2([M]; \Z /2) / \{((1/2)tdt, dt)\},$$ where the group product is $(w,p)(v,q) = (w+v+(1/2)p\cup_1q, p+q)$.  Therefore homomorphisms $G^1 \to \R / \Z$ are all of the form $I(w) + Q(p)$, where $I$ is continuous and linear and $Q$ satisfies the two conditions of the proposition. So $I$ is just multiplication by some integer.   If one $Q$ is given, then all other $Q$ are obtained by adding a linear function of $p$ that vanishes on boundaries $dt$.  But this just means a linear function  $H^2(M; \Z /2) \to \Z /2$, which is the same thing as multiplying $p$ by a cohomology class $b \in H^1([M]; \Z/2)$.\\

The evaluation of $(w,p,0)$ on the identity $M = M$ is given by $I_0 + Q_0$, where $I_0(w) = \int_{[M]} w$ and $Q_0(p) = (1/2)\int_{[W]} P^2$.  Here, $[W]$ is an ordered simplicial 4-manifold with boundary $[M]$ and $P$ is a cocycle lift of $p$ to $[W]$. Or, write $p + dt = p_u$ where $p_u = z^*(u)$ for some ordered simplicial map $z\colon M \to S^2$.  Then $Q_0(p) = (1/2)[Spin(z)] + (1/2)\int_{[M]} (tdt + p\cup_1 dt)$.\\ 

Consider the evaluations of $(w,p,0)$ on $M' = M$, where $M'$ is another Spin structure. Since the map $M' = M$ is just the identity, $I(w) = \int_{[M]} w = I_0(w)$ is the same for all $M'$. The very simple $G^1$ case of Proposition~\ref{deltaspin} implies that if $M' = M_\beta$, then $Q(p) = Q_0(p) + (1/2)pb$.\\

Therefore we see that as the Spin structure on $M$ is varied, we see exactly once all possible quadratic functions $Q$ satisfying the conditions of the proposition.  This completes the proof.
\end{proof}

The characterization of Spin structures on 3-manifolds given in Proposition 9.4 is an analogue of the result that $\text{Pin}^-$ structures on surfaces canonically correspond to quadratic functions  $q\colon  H^1(\ZZ / 2) \to \Ar / \ZZ$. Quadratic functions and Spin structures in both cases can be parametrized by $H^1(\ZZ / 2)$, but such parametrizations are not canonical.  
  
  \section[Appendix B]{The Proof of Proposition~\protect{\ref{3.6}}}\label{A.3}

We will now describe how to evaluate the element $((1/2)c^3 + (1/4)C^3, c^2, 0) \in G(T_SS^2)$
 on  (the identity map of) $T_SS^2$, where $c$ represents the generator of $ H^1(T_SS^2; \ZZ/2)$ and where $T_SS^2$   is given a Spin structure as the boundary of $T_DS^2$, the tangent unit disk bundle of $S^2$.  Recall this computation is a key step (Proposition~\ref{3.6}) in the proof that our pairing $G \times \Omega^{spin}_3 \to \RR /\ZZ$ is linear in the $G$ variable.
Originally,
we used computer calculations of Jacek Skryzalin (a student of Gunnar Carlsson at Stanford) to make this evaluation. Subsequently, we were able to do this by hand with a simpler cocycle $c$ and that is what is presented in this section.  \\

Here is the proposition that remains to be proved.

\begin{propa}
Equip $T_{S}S^2$ with the Spin structure it inherits as the boundary of the Spin manifold $T_{D}S^2$
associated with its natural orientation.
 Let $c\in Z^1(T_{S}S^2;\Zee/2)$ represent the non-trivial cohomology class. Then
  the class in $G^1(T_{S}S^2)$ represented by 
$$((1/2)c(c\cup_1c)c+(1/4)C^3,c^2,0)$$
evaluates to $1/4$ on the identity map considered as an element of $\Omega_3^{spin}(T_SS^2)$.
\end{propa}

\begin{proof}
 Let $H^\pm$ be the closed upper and lower hemispheres of $S^2$. We write the tangent circle bundle as the union of $H^+\times S^1 \cup H^-\times S^1$.
 In each factor we call the fiber direction {\em vertical} and the $H^\pm$ direction {\em horizontal}.
We identity $H^\pm$ with the unit disk in the complex plan via stereographic projection from the south/north pole,
and we use the orientation of the fiber circle coming from the counter-clockwise direction of the unit circle in the tangent plan at the origin, which stereographic projection identifies with $\Cee$. Then the orientation of $T_SS^2$ agrees with the product of these orientations on $H^\pm\times S^1$. Their  boundaries are oriented as $\partial H^\pm\times S^1$, each of which we identify with $S^1\times S^1$ with the product of the standard orientations on each factor.  The gluing map
$$\begin{CD}
 \partial H^-\times S^1 @>{\mu}>>  \partial H^+\times S^1 \\
 @V=VV @VV=V \\
 S^1\times S^1 @>{\nu}>> S^1\times S^1
\end{CD}$$ 
is given by $\nu(z_1,z_2)=(z_1,z_1^2z_2)$.

  We triangulate $S^2$ by first triangulating the equator with eight vertices
labeled $0,\ldots,7$ in the counterclockwise cyclic order around $\partial H^+$. We then add the north pole, labeled $n$, and the south pole, labeled $s$, and cone the triangulation of the equator to each of these additional vertices.

We triangulate  $\partial H^-\times  S^1$ as indicated in Figure 1 below. In this figure the opposite sides of each rectangle are identified to form
the torus; the horizontal direction is base direction and the vertical direction is the fiber direction. The projection to the equator is simplicial.
Using $\mu$ we transfer  the triangulation to one for $\partial H^+\times  S^1$.
It is given in Figure 2.  (Notice that since the rectangle is $8\times 4$ a curve of slope $2$ in $S^1\times S^1$ has slope one in the Cartesian coordinates of  this representation.)
We label the vertices of $\partial H^+\times S^1$ by the Cartesian coordinates of Figure 2. Thus, the vertices are given by pairs $(x,y)$ with $x$ and $y$ integers with $x$ considered modulo $8$ and $y$ modulo $4$. The $y$-coordinate is called the {\em level}.
We extend the triangulation of the circles at level $0$ and $2$ to a triangulation of the copy of $H^+$ at level $0$ and $2$ by coning the triangulation on the boundary to the point $(n,0)$ and $(n,2)$ which are the images of the north pole at levels $0$ and $2$. These disks divide $$H^+\times  S^1$$ into
two $3$-balls whose boundaries are triangulated. We extend the triangulation over each of these balls by coning to the central points $(n,1)$ and $(n,3)$,
the images of the north pole at levels $1$ and $3$. The projection mapping of $H^+\times S^1\to H^+$ is simplicial.

\vskip.5in

\begin{tikzpicture}
\foreach \x in {-6,-4.5,-3,-1.5,0,1.5,3,4.5,6} \draw (\x cm, -3 cm) -- (\x cm, 3 cm);
\foreach \y in {-3,-1.5,0,1.5,3} \draw (-6 cm,\y cm) -- (6 cm, \y cm);
\foreach \x in {-6, -3, 0} \draw (\x cm, -3 cm) -- (\x cm + 6 cm, 3 cm);
\foreach \x in {-6, -3, 0}
\draw (\x cm , 3 cm) -- (\x cm + 6 cm, -3 cm);
\draw (-6,0) -- (-3,3);
\draw (3, -3) -- (6,0);
\draw(3, 3) -- (6,0);
\draw (-6,0) -- (-3,-3);
\foreach \x in {-6,-4.5,-3,-1.5,0,1.5,3,4.5,6}
\draw[ultra thick] (\x cm ,-3 cm) -- (\x cm , -1.5 cm);
\foreach \x in {-6,-3,0,3} \draw[ultra thick] (\x cm, -3 cm) -- (\x cm+ 1.5 cm,-1.5 cm);
\foreach \x in {-4.5,-1.5,1.5,4.5} \draw[ultra thick] (\x cm,-1.5 cm) -- (\x cm+1.5 cm,-3 cm);
\draw (0,-4.5 cm) -- (0,-4.5 cm) node[anchor=north] {Figure 1. $\partial H^-\times S^1$ with support of cocycle $c_0$ marked in dark black};
\end{tikzpicture}

\vskip .5in

\begin{tikzpicture}

\foreach \y in {-3,0,3} \draw (-6 cm, \y cm) -- (6 cm, \y cm);
\foreach \x in{-6, -4.5,-3,-1.5,0,1.5,3,4.5,6}
\draw (\x cm, -3 cm) -- (\x cm,3 cm);
\foreach \x in {-6,-4.5,-3,-1.5,0} \draw (\x cm,-3cm)-- (\x cm+ 6 cm,3 cm);
\foreach \x in {1.5,3,4.5} \draw (\x cm,-3 cm) -- (6 cm, 3 cm-\x cm);
\foreach \y in {-1.5,0,1.5} \draw (-6 cm,\y cm ) -- (-3 cm - \y cm,3 cm);
\foreach \x in {-6,-4.5,-3,-1.5,0,1.5,3} \draw (\x cm, -3 cm) -- (\x cm+ 3 cm, 3 cm);
\draw (4.5 cm,-3 cm) -- (6 cm, 0 cm);
\draw (-6 cm, 0 cm) -- (-4.5 cm, 3 cm);
\foreach \x in {-6 ,0, 6} \draw[ultra thick] (\x cm, -3 cm) -- (\x cm, -1.5 cm);
\foreach \x in {-4.5,1.5} \draw[ultra thick] (\x cm, -1.5 cm) -- (\x cm, 0 cm);
\foreach \x in {-3,3} \draw[ultra thick]  (\x cm, 0 cm) -- (\x cm , 1.5 cm);
\foreach \x in {-1.5, 4.5} \draw[ultra thick] (\x cm ,1.5 cm) -- (\x cm, 3 cm);
\foreach \x in {-4.5,1.5} \draw[ultra thick] (\x cm , 0 cm) -- (\x cm+1.5 cm, 0 cm);
\foreach \x in {-6,0} \draw[ultra thick] (\x cm, -3 cm) -- (\x cm +1.5 cm, 0 cm);
\foreach \x in { -3, 3} \draw[ultra thick] (\x cm,0 cm) -- (\x cm+1.5 cm, 3 cm);
\foreach \x in {-1.5,4.5} \draw[ultra thick] (\x cm, -3 cm )-- (\x cm +1.5 cm,-3cm);
\foreach \x in {-1.5,4.5} \draw[ultra thick] (\x cm, 3 cm) -- (\x cm+1.5 cm,3 cm);
\foreach \x/\xtext in {-6/{(0,0)},-4.5/{(1,0)},-3/{(2,0)},-1.5/{(3,0)},0/{(4,0)},1.5/{(5,0)},3/{(6,0)},4.5/{(7,0)},5.96/{(8,0)}}
\draw (\x cm, -3 cm+1.5pt) -- (\x cm,-3cm-1.5pt) node[below=1.5pt] {$\xtext$};
\foreach \y/\ytext in {-1.5/{(0,1)}, 0/{(0,2)}, 1.5/{(0,3)}, 3/{(0,4)}}
\draw (-6.2cm, \y cm)-- (-6.2 cm, \y cm) node[anchor=east] {$\ytext$};
\draw (0,-4.5 cm) -- (0,-4.5 cm) node[anchor=north] {Figure 2. $\partial H^+\times S^1$ with support of cocycle $c_0$ marked in dark black};
\end{tikzpicture}

\vskip.5in

There is an analogous construction to extend the triangulation over $H^-\times S^1$, adding new vertices at $(s,0), (s,2)$ and then $(s,1)$ and $(s,3)$. The projection mapping to $H^-$ is simplicial. (Of course, the horizontal disks at levels $0$ or $2$ on the two sides do not match along the boundary since the boundary identification does not preserve level.)

Now we define the reduced simplicial $1$-cocycle $c$. We begin with a cocycle $c_0$ on $\partial H^-\times S^1$ as indicated in black in Figure 1. It evaluates non-zero on every edge with one vertex at level $0$ and one vertex at level $1$, and evaluates $0$ on all other edges of $H^-\times S^1$. We extend this to a reduced simplicial cochain $c^-$ on
$H^-\times S^1$ by setting its value equal to $1$ on any edge with one vertex of level $0$  and one vertex of level $1$ and $0$ on all other edges. Since any $2$-simplex in $H^-\times S^1$ has vertices at at most $2$ levels, it is clear that $c^-$ is a reduced $1$-cocycle.

The image of $c_0$ under $\mu$ on $\partial H^+\times S^1$ is indicated  in black in Figure 2. Its extension over $H^+\times S^1$, denoted $c^+$,  is more complicated.
In addition to the boundary edges indicated in Figure 2, the extended reduced simpicial cochain is non-zero exactly on the following interior edges:
\begin{itemize}
\item $\{(n,0),(n,1)\}$
\item $\{(x,0),(n,0)\}$ for $4\le x\le 7$
\item $\{((x,0)),(n,1)\}$ for $0\le x\le 3$
\item $\{x,1),(n,1)\}$ for $1\le x\le 4$
\item $\{(x,2),(n,1)\}$ for $2\le x\le 5$
\item  $\{(x,2),(n,2)\}$ for $2\le x\le 5$
\item $\{(x,2),(n,3)\}$ for $2\le x\le 5$
\item $\{(x,3),(n,3)\}$ for $3\le x\le 6$
\item $\{(x,4),(n,3)\}$ for $4\le x\le 7$
\end{itemize}

Either a direct inspection of the $3$-dimensional picture of the triangulation of $H^+\times S^1$ described above, or a computer computation, shows that for each $2$-simplex, $c^+$ is  non-zero on either $0$ or $2$ of its edges, and hence $c^+$ is a cocycle.

We denote by $c$ the reduced simplicial cocycle that agrees with $c^\pm$ on $H^\pm\times S^1$. Clearly, this cocycle evaluates $1$ on the fiber circle,

From now on we use a different labelling of the vertices, one in which the partial ordered (which is needed to compute products) is evident. Referring to Figure 2, the vertices on $\partial H^+\times S^1$
 are labelled by $AB$ where $0\le A\le 3$ is congruent to $x+y$ modulo $4$ and $0\le B\le 7$ is $2y$ if $x\le 3$ and is $2y+1$ if $4\le x\le 7$.
Vertices $(n,y)$ in $H^+\times S^1$ are labelled $AB$ with $A=4$ if $y$ is even and $A=5$ if $y$ is odd and with $B=2y$.
 The vertices $(s,y)$ in $H^-\times S^1$ are labelled $AB$ with $A=6$ if $y$ is even and $A=7$ if $y$ is odd and with $B=2y$.
 The partial order of these vertices is given by the function to $\{0,\ldots,7\}$ given by $AB\mapsto A$.

The only possible $2$ simplices in the support of $(c^-)^2$ must be cones to $(s,1)$ of edges in the support of $c_0$. But each of the
 edges in the support of $c_0$  has the property that its vertex at level zero is less (in the given partial order) than its vertex at level $1$. Thus, $(c^-)^2$ vanishes on all these $2$-simplices, and hence $(c^-)^2=0$ 
 
 Direct inspection of the $3$-dimensional picture of the triangulation of $H^+\times S^1$, or computer computation, shows that 
$(c^+)^2$ is non-zero on exactly the following $2$-simplices:
\begin{align*} & \{ (00),(31),(40)\},\{(31),(40),(52)\},
\{(21),(40),(52)\},
\{(11),(40),(52)\}, \\ 
& \{(01),(40),(52)\},\{(01),(30),(52)\},\{(01),(13),(52)\},
\{(01),(35),(52)\}, \\
 & \{(23),(35),(52)\},\{(05),(35),(52)\},\{(05),(35),(44)\},\{(05),(35),(56)\},\\
&\{(05),(17),(56)\},
 \{(05),(31)),(56)\},\{(27),(31),(56)\},\{(00),(31),(56)\}\hfill
\end{align*}

Lastly, $C^3$ is non-zero on only one  $3$-simplex, namely $\{(00),(31),(40),(52)\}$, where it takes value $1$.
Since the orientation of this simplex induced by its vertex ordering disagrees with the restriction of the ambient orientation of $T_{S}S^2$, it follows that
$$\int_{T_{S}S^2}C^3=-1.$$

Next, we introduce the reduced simplicial $2$-cocylce $p$ which is the pullback under the projection mapping of the the reducecd simplicial $2$-cocycle on $S^2$ evaluating non-zero only on the
$2$-simplex $\{6,7,n\}$ of the base $S^2$. The cocycle $p$ is non-zero only on the following $2$-simplices:
\begin{align*}
&\{(21),(31),(40)\},\{(21),(31),(52)\},\{(03),(21),(52)\},\{(15),(21),(52)\}, \\
&\{(15),(33),(52)\},\{(05),(15),(52)\},\{(05),(15),(44)\},\{(05),(15),(56)\},\\
&\{(05),(27),(56)\},\{(05),((31),(56)\},\{(17),(31),(56)\},\{(21),(31),(56)\}
\end{align*}

Set $t$ equal to the reduced $1$-cochain that evaluates non-zero exactly on the following edges:
\begin{align*}
&\{(31),(40)\},\{(21),(52)\},\{(11),(52)\},\{(01),(52)\},\{(23),(52)\}, \\
&\{(33),(52)\},\{(05),(52)\},\{(05),(44)\},\{(05),(56)\},\{(31),(56)\}
\end{align*}

Direct inspection of the $3$-dimensional picture of the triangulation of $H^+\times S^1$, or computer computation,  shows that $p+dt=c^2$ and that $tdt=tp=pt=0$. 
Thus, we have 
$$(0,p,0)((1/2)tdt,dt,0)=((1/2)(tdt+p\cup_1dt),p+dt,0)=((1/2)(p\cup_1dt),c^2,0).$$
Also, 
$$d(p\cup_1t)=p\cup_1dt+pt+tp=p\cup_1dt,$$
so that in $G(T_SS^2)$ we have
$$(0,p,0)=(0,c^2,0),$$
and thus
$$(\frac{1}{2}c^3+\frac{1}{4}C^3,c^2,0)=(\frac{1}{2}c^3+\frac{1}{4}C^3,p,0)$$
as elements of $G(T_SS^2)$.
Since $p$ is pulled back from a reduced $2$-cocycle on $S^2$ via the projection map,  it follows
from Defintion~\ref{defn1} that the value of
$(\frac{1}{2}c^3+\frac{1}{4}C^3,p,0)$ on the identity map of $T_SS^2$ to itself is equal to 
$$\frac{1}{2}[Spin(p)]+\int_{T_SS^2}\bigl(\frac{1}{2}c^3+\frac{1}{4}C^3\bigr).$$
Since the projection map extends over $T_DS^2$, 
the Spin structure on the preimage of a generic point (i.e., on a fiber) bounds the Spin manifold that is the corresponding tangent disk in $T_DS^2$. Thus, $[Spin(p)]=0$,
and hence
$$\langle (\frac{1}{2}c^3+\frac{1}{4}C^3,p,0),(T_SS^2,{\rm Id}_{T_SS^2})\rangle = \int_{T_SS^2}\bigl(\frac{1}{2}c^3+\frac{1}{4}C^3\bigr)=\frac{1}{2}-\frac{1}{4}=\frac{1}{4}.$$

This completes the proof of the proposition.
\end{proof}

\section{Cochain Formulas}\label{sectcochain}

\subsection{Cup and $\text{Cup}_i$ Products}

The Alexander-Whitney diagonal approximation and the associated higher homotopies are usually defined in the context of singular chains and cochains. The same formulas hold for ordered simplicial cochains. Using these one obtains the usual specific simplicial cochain formulas for cup products, denoted $\cup$ or $\cup_0$, and also for $\cup_i $ products for $i > 0$ in either context. The $\cup_i$ product is a bilinear pairing of integral cochains $$\cup_i: C^m(\Z) \times C^n(\Z) \to C^{m+n-i}(\Z).$$ It is well-known that the Alexander-Whitney diagonal approximation is just one specific example of a diagonal approximation.  Of course it is historically important, and fairly easy to remember, but any diagonal approximation leads to a theory of cup products.  Higher homotopies associated to a diagonal approximation, which lead to the $\text{cup}_i$ products,
are also not unique.   But different choices lead to chain homotopic maps of tensor product complexes, hence the cohomology operations underlying different choices are the same.\\
 
We need a few formulas involving $\cup_i$ products. The proofs require getting inside the mechanism for constructing $\cup_i$ products from the Alexander-Whitney diagonal and higher homotopies, which we will not develop here. \\

(CP1) Coboundary formula: For $X \in C^m(\Z)$ and $Y \in C^n(\Z)$
  $$d(X \cup_i Y) = (-1)^i\ dX\cup_iY + (-1)^{i+m}\ X \cup_idY - (-1)^i\ X\cup_{i-1}Y - (-1)^{mn}Y\cup_{i-1}X$$
  
Next, we give formulas for evaluating  integral cochain $\cup_1$ and $\cup_2$ products on simplices of dimension 1, 2, and 3. We let $(0,1,\cdots, k)$ denote standard simplices.  The $i^{th}$ face is obtained by deleting the $i^{th}$ vertex.  For example $(0,2)$ means the face of the 2-simplex $(0,1,2)$ opposite vertex 1. The products and sums on the right below mean arithmetic in $\Z$ with the values of cochains on simplices.\\

For$A,B \in C^1(\Z)$ and $P,Q \in C^2(\Z)$\\ 

(CP2)  $A\cup_1B (0,1) = (-1)A(0,1)B(0,1)$\\

(CP3) $A\cup_1P(0,1,2) = (-1)A(0,2) P(0,1,2)$\\

(CP4) $P\cup_1(0,1,2) = P(0,1,2)A(0,1) + P(0,1,2)A(1,2))$\\

(CP5) $P\cup_1Q(0,1,2,3) = P(0,1,3)Q(1,2,3) - P(0,2,3)Q(0,1,2)$\\

(CP6) $P\cup_2Q(0,1,2) = (-1)P(0,1,2)Q(0,1,2)$\\

Of course all these formulas hold for $\Z/2$ cochains $a,b,p,q,x,y \in C^*(\Z/2)$ and a coccyle $z \in C^1(\Z/2)$, with the  advantage that we can ignore $\pm$ signs, so one gets simplifications. 

\subsection{Special Integral Lifts of $\Z/2$ Cochins}

 Consider the coefficient homomorphism $\Z \to \Z/2 = \{0, 1\}$. Of course this induces a group homomorphism $C^k(\Z) \to C^k(\Z/2)$. There is the obvious set theoretic splitting $\Z/2 \to \Z$ given by $0 \mapsto 0$ and $1 \mapsto 1$.  For any abelian coefficient group $M$, $C^k(M) = Functions (X(k), M)$, where $X(k)$ is the basis of the free abelian chain group $C_k(\ZZ)$ given by the $k$-simplices. The set map $\Z/2 \to \Z$ above thus determines a function $C^k(\Z/2) \to C^k(\Z)$, which we call the $\it{special\ lift}$  map. Special lifts of cochains, which only take values 0 and 1 on simplices, turn out to have some very nice properties because the arithmetic of 0 and 1 in $\Z$ is so simple.\\
 
Given $x \in C^k(\Z/2)$, we denote the special lift by the capital letter $X \in C^k(\Z)$.  We denote the special lift of the coboundary $dx \in C^{k+1}(\Z/2)$ by the symbol $Dx \in C^{k+1}(\Z)$.  Note $Dx \not= dX$.\\

(SL1) If $A$ is the special lift of a cocycle $a \in Z^1(\Z/2)$ then $dA = 2A^2$.\\

To see this, just observe that $A^2(0,1,2) = 1$ if and only if $a(0,1) = a(1,2) = 1$ and $a(0,2) = 0$. (Note that since $a$ is a cocycle, $dA$ only takes values 0 and 2, never 1 or $-1$.)\\

(SL2) If $A, B$ are special lifts of $a,b \in Z^1(\Z/2)$ then the special lift of $c = a+b$ is $C = A + B + 2(A\cup_1B)$.\\

$A\cup_1B(0,1) = -1$ if and only if $a(0,1) = b(0,1) = 1$. So $C(0,1) = 0$ if  $a(0,1) = b(0,1)$, and $C(0,1) = 1$ if $a(0,1) \not= b(0,1)$.  This says $C$ is the special lift of $a+b$.\\

(SL3) If $C$ is the special lift of $a+b$, as in (SL2), then $C^2 = A^2 + B^2 + d(A\cup_1B)$.\\

The easiest way to see this is to look at $$ 2C^2 = dC  = dA + dB + 2d(A\cup_1B) = 2(A^2+B^2 +  d(A\cup_1B))$$   
This is an equation in the torsion free group $C^2(\Z)$, so we can divide by 2.\\

(SL4) If $P$ is the special lift of a cocycle $p \in Z^2(\Z/2)$ then $dP = 2(P\cup_1P)$.  Consequently, in $C^3(\R/\Z)$, we have $d(P/4) = (1/2)(p\cup_1p)$.\\

From (CP5), $P\cup_1P(0,1,2,3) = P(0,1,3)P(1,2,3) - P(0,2,3)P(0,1,2)$. Since $p$ is a $\Z/2$ cocycle, it takes value $1 \in \Z/2$ on either zero, two, or four faces of $(0,1,2,3)$. The special lift $P$ will take value 1 on these same faces. One sees that $dP$ takes values 0, 2, or -2 depending on which faces $p$ is non-zero.  Specifically, $dP$ takes value 2  if $p$ is 1 exactly on the faces $(0,1,3), (1,2,3)$, and   $dP$ takes value $-2$ if $p$ is 1 exactly on the faces  $(0,2,3)$, $(0,1,2)$.  In all other cases, $dP$ takes value 0.  But this inspection of all cases agrees with the evaluation of $2(P\cup_1P)$.\\

SL(1) and SL(4) generalize to $dQ \equiv 2(Q\cup_{n-1}Q) \pmod 4$ for any $\ZZ/2$ n-cocycle $q$.  An application is that the integral Bockstein of a $\ZZ/2$ cocycle $q$ reduces modulo 2 to $Sq^1(q)$.  Although this is well-known, it is not trivial to prove.\\

(SL5)  If $A$ is the special lift of $a \in Z^1(\Z/2)$ then $A^2 \cup_1 A^2 = 0$\\

Since $A^2$ is an integral cocycle, SL(4) implies $2(A^2 \cup_1A^2)= 0$.  Since the integral cochain groups are torsion free, SL(5) is proved.\\

Finally we look at a delicate result involving cochains $x \in C^0(\Z/2)$. We have the special lifts $X \in C^0(\Z)$ of $x$ and $Dx\in C^1(\Z)$ of $dx$. We also have $dX \in C^1(\Z)$, which can take values 0, 1, or $-1$. The cup products $(Dx)X$ and $X(Dx)$ are special lifts of $dx\ x$ and $x\ dx$, respectively.\\

(SL6)  $Dx + dX = 2(Dx)X$ and $Dx - dX = 2X(Dx)$, both equalities holding in $C^1(\Z)$.\\

The proof is by a simple analysis of cases, depending on the  values $x(0), x(1)$.

\end{document}